\documentclass[11pt]{amsart}

\usepackage{hyperref}
\usepackage[english]{babel}
\usepackage{amsmath,amssymb,amsthm}
\usepackage[pdftex]{graphics,color}
\usepackage{epsfig}
\usepackage{pstricks}

\newtheorem{theorem}{Theorem}[section]
\newtheorem{proposition}[theorem]{Proposition}
\newtheorem{lemma}[theorem]{Lemma}
\newtheorem{corollary}[theorem]{Corollary}

\newtheorem{question}[theorem]{Question}

\theoremstyle{definition}
\newtheorem{definition}[theorem]{Definition}

\newtheorem*{remark}{Remark}

\newcommand{\Z}{\mathbb{Z}}
\newcommand{\N}{\mathbb{N}}
\newcommand{\R}{\mathbb{R}}
\renewcommand{\P}{\mathbb{P}}
\newcommand{\E}{\mathrm{\textbf{E}}}
\newcommand{\F}{\mathcal{F}}
\newcommand{\G}{\mathcal{G}}
\renewcommand{\L}{\mathcal{L}}

\newcommand{\C}{\mathbb{C}}

\renewcommand{\H}{\mathbb{H}}

\newcommand{\dist}{\operatorname{dist}}
\renewcommand{\Im}{\textrm{Im}}
\renewcommand{\Re}{\textrm{Re}}
\newcommand{\SLE}{\textrm{SLE}}
\newcommand{\SLEk}{\textrm{SLE}(\kappa)}

\newcommand{\Gt}{\G_t}
\newcommand{\prob}[1]{\P \left ( #1 \right)}
\newcommand{\probalph}[2]{\P_{#2} \left( #1 \right)}

\newcommand{\palph}{\P_{\alpha}}
\newcommand{\expect}[1]{\mathrm{\textbf{E}} \left[ #1 \right]}
\newcommand{\subexpect}[3]{\mathrm{\textbf{E}}_{#2}^{#3} \left[ #1 \right]}

\newcommand{\condsubexpect}[4]{\mathrm{\textbf{E}}_{#3}^{#4} \left[ \left. #1 \right| #2 \right]}
\newcommand{\indicate}[1]{\mathbf{1} \left \{ #1 \right \}}

\newcommand{\ep}[1]{#1_{\epsilon}}
\newcommand{\haussdim}[0]{\mathrm{dim_H} \,}
\newcommand{\Eze}{I(z, \epsilon) \cap K = \emptyset}
\newcommand{\EE}[2]{I(#1, #2) \cap K = \emptyset}
\newcommand{\phize}{\phi_{I(z, \epsilon)}}
\newcommand{\fze}{f_{z, \epsilon}}
\newcommand{\ls}{\L_{\lambda}}
\newcommand{\leb}[1]{\left| #1 \right|}

\begin{document}

\title{Bridge Decomposition of Restriction Measures}
\author{Tom Alberts}
\address{Department of Mathematics \\
University of Toronto \\
Toronto, ON, Canada}
\thanks{Research of Tom Alberts supported in part by NSF Grant OISE 0730136, and a postdoctoral fellowship from the Natural Sciences and Engineering Research Council of Canada. Research of Hugo Duminil-Copin supported in part by project MRTN-CT-2006-035651, Acronym CODY, of the European Commission, and a grant from the Swiss National Science Foundation.}
\email{alberts@math.toronto.edu}

\author{Hugo Duminil-Copin}
\address{\'Ecole Normale Sup\'erieure \\
Paris, France \\}
\email{hugo.duminil@ens.fr}

\date{}

\begin{abstract}
Motivated by Kesten's bridge decomposition for two-dimensional self-avoiding walks in the upper half plane, we show that the conjectured scaling limit of the half-plane SAW, the $\SLE(8/3)$ process, also has an appropriately defined bridge decomposition. This continuum decomposition turns out to entirely be a consequence of the restriction property of $\SLE(8/3)$, and as a result can be generalized to the wider class of \textit{restriction measures}. Specifically we show that the restriction hulls with index less than one can be decomposed into a Poisson Point Process of \textit{irreducible bridges} in a way that is similar to It\^{o}'s excursion decomposition of a Brownian motion according to its zeros.
\end{abstract}

\maketitle

\section{Introduction}

One of the greatest successes of the Schramm-Loewner Evolution (SLE), and the broader study of two-dimensional conformally invariant stochastic processes that it enabled, has been the ability to describe the scaling limits of two-dimensional lattice models that arise in statistical mechanics. There are many known examples: $\SLE(2)$ as the scaling limit for loop erased random walk, $\SLE(3)$ as the scaling limit of critical Ising interfaces, $\SLE(6)$ as the limit of percolation exploration paths, etc. One of the most important open problems in the field is to prove that the scaling limit of the infinite \textit{self-avoiding walk} in the upper half plane $\H$ is given by $\SLE(8/3)$. It is known that \textit{if} the scaling limit of half-plane SAWs exists \textit{and} is conformally invariant, then the scaling limit must be SLE($8/3$). Both the existence and conformal invariance are widely believed to be true, yet proofs remain elusive. For an accessible and relatively recent source on the current status of this problem, we refer the reader to \cite{lsw:saw}. Even without formally establishing the scaling limit result, it is often still possible to independently check that the various well-studied properties of half-plane SAWs carry over to the SLE($8/3$) process. The main results of this paper should be seen in this context. In \cite{kesten:saw1} it is shown that half-plane SAWs admit what is called a \textit{bridge decomposition}, which raised the question of finding a similar decomposition for SLE($8/3$). In this paper we will show that an appropriately defined continuum decomposition does exist, and we will describe some of its properties. A somewhat surprising aspect of the existence is that it depends only on the fact that SLE($8/3$) satisfies the restriction property, and not on the fine details of the process itself. Specifically, the decomposition has no explicit reliance on the Loewner equation. Using this fact we are able to extend the continuum bridge decomposition beyond SLE($8/3$) to a wider class of random sets whose laws are given by the so-called restriction measures. These probability measures were introduced and studied extensively in \cite{lsw:conformal_restriction}, and they occupy an important position in the hierarchy of two-dimensional conformally invariant processes. We will give a more detailed description of restriction measures in Section \ref{PrelimSection}, but we emphasize that the reader who is uninterested in general restriction measures will lose nothing by focusing on SLE($8/3$) as the canonical one.

\subsection{Motivation: Bridge Decomposition of SAWs}

To motivate the continuum bridge decomposition, we first describe the corresponding decomposition for half-plane SAWs. This is thoroughly described in \cite{madras_slade:saw_book}, along with many other interesting properties of the self-avoiding walk. In the discrete setting we will work exclusively on the lattice $\Z + i \Z$. An $N$-step self-avoiding walk $\omega$ on $\Z + i\Z$ is a sequence of lattice sites $[\omega(0), \omega(1), \ldots, \omega(N)]$ satisfying $\leb{\omega(j+1) - \omega(j)} = 1$ and $\omega(i) \neq \omega(j)$ for $i \neq j$. We will write $\leb{\omega} = N$ to denote the length of $\omega$. Given walks $\omega$ and $\omega'$ of length $N$ and $M$ (respectively), the \textit{concatenation} of $\omega$ and $\omega'$ is defined by
\begin{align*}
\omega \oplus \omega' = \left[ \omega(0), \ldots, \omega(N), \omega'(1) + \omega(N), \ldots, \omega'(M) + \omega(N) \right].
\end{align*}
Letting $c_N$ denote the number of self-avoiding walks of length $N$, it is easy to see that
\begin{align*}
c_{N+M} \leq c_N c_M
\end{align*}
since any SAW of length $N+M$ can always be written as the concatenation of two SAWs of length $N$ and $M$. A standard submultiplicativity argument then proves the existence of a constant $\mu > 0$ such that
\begin{align}\label{connectiveConstant}
\lim_{N \to \infty} \frac{\log c_N}{N} = \log \mu,
\end{align}
or $c_N \approx \mu^N$ in the common shorthand. The exact value of $\mu$ is not known, nor is it expected to be any special value, but numerically it has been shown that $\mu$ is close to $2.638$ (see \cite[Section 1.2]{madras_slade:saw_book}).

We will mostly deal with half-plane SAWs rooted at the origin, i.e. self-avoiding paths $\omega$ such that $\omega(0) = 0$ and $\Im \, \omega(j) > 0$ for all $j > 0$. Let $\mathcal{H}$ denote the set of all such walks. The most commonly used probability measure on $\mathcal{H}$, and the one that we will consider throughout, is the weak limit of the uniform measure on $\{ \omega \in \mathcal{H} : \leb{\omega} =N \}$, as $N \to \infty$. This limit is proven to exist in \cite{madras_slade:saw_book}, and again in the appendix of \cite{lsw:saw}. The key element of both proofs is, in fact, the \textit{bridge decomposition} of the walks in $\mathcal{H}$, the study of which was initiated by Kesten \cite{kesten:saw1, kesten:saw2} and goes as follows. A \textit{bridge of length $N$} is a self-avoiding walk $\omega$ such that $\leb{\omega} = N$ and
\begin{align*}
\Im \, \omega(0) < \Im \, \omega(j) \leq \Im \, \omega(N), \quad 1 \leq j \leq N.
\end{align*}
Note that the concatenation of any two bridges is still a bridge, but that not every bridge is the concatenation of two shorter ones. A bridge with the latter property is said to be \textit{irreducible}, and such bridges are the basic building blocks of walks in $\mathcal{H}$. Indeed, given any $\omega \in \mathcal{H}$, one performs a bridge decomposition of $\omega$ by searching for the smallest time $j$ such that $\Im \, \omega(k) \leq \Im \, \omega(j)$ for $k \leq j$ and $\Im \, \omega(k) > \Im \, \omega(j)$ for $k > j$. By the minimality of $j$, the subpath $[w(0), w(1), \ldots, w(j)]$ is an irreducible bridge, and the shifted subpath $[0, w(j+1) - w(j), \ldots, w(k) - w(j), \ldots]$ for $k \geq j$ is a new element of $\mathcal{H}$ on which we may repeat this procedure. Iterating in this fashion produces the bridge decomposition of $\omega$ into a sequence of irreducible bridges, and the decomposition is clearly unique\footnote{There is a minor technicality to point out here: if the walk oscillates infinitely often in the vertical direction without approaching some limit (including infinity) the decomposition algorithm will terminate after finitely many iterations and the remaining part of the walk will not be a bridge. However, we will see in the next paragraph that this is a probability zero event under the standard measure on $\mathcal{H}$, and that the vertical component of the SAW always goes to infinity with probability one.}.

Much of the study of the infinite self-avoiding walk in the upper half plane therefore reduces to the study of irreducible bridges. Let $\mathcal{B}$ be the set of all irreducible bridges rooted at the origin, and $\lambda_N$ be the number of length $N$ elements of $\mathcal{B}$. Using some clever tricks involving generating functions, Kesten was able to prove what is now called \textbf{Kesten's relation}:
\begin{align}\label{kesten1}
\sum_{N \geq 1} \lambda_N \mu^{-N} = \sum_{\omega \in \mathcal{B}} \mu^{-\leb{\omega}} = 1,
\end{align}
for the same $\mu$ as in \eqref{connectiveConstant} (for proofs see \cite{kesten:saw1} or \cite[Section 4.3]{madras_slade:saw_book}). Kesten's relation shows that $\textbf{P}(\omega) := \mu^{-\leb{\omega}}$ is a probability measure on $\mathcal{B}$, and by concatenating together an independent sequence of irreducible bridges each sampled from $\textbf{P}$, a probability measure is induced on $\mathcal{H}$. In \cite{madras_slade:saw_book} and \cite{lsw:saw}, the latter measure is shown to be the only possible candidate for the weak limit of the uniform measure on $\{ \omega \in \mathcal{H} : \leb{\omega} =N \}$, and therefore the question of existence of this weak limit is immediately settled.

The bridge decomposition shows that infinite half-plane SAWs have a renewal structure to them. At the end of each irreducible bridge the future path of the walk lies entirely in the half-plane above the horizontal line where the bridge ended. The future path is again a concatenation of a sequence of irreducible bridges, so that its law is the same as the law of the original path and the future path is independent of the past. In this sense the walk renews itself whenever it is at the end of an irreducible bridge, and it is appropriate to call such times renewal times. Note that the renewal times are functions of the \textit{entire} half-plane SAW, since the algorithm for the bridge decomposition depends upon knowing the entire walk.

\begin{figure}
\begin{center}
\includegraphics[width=15cm, height=12cm]{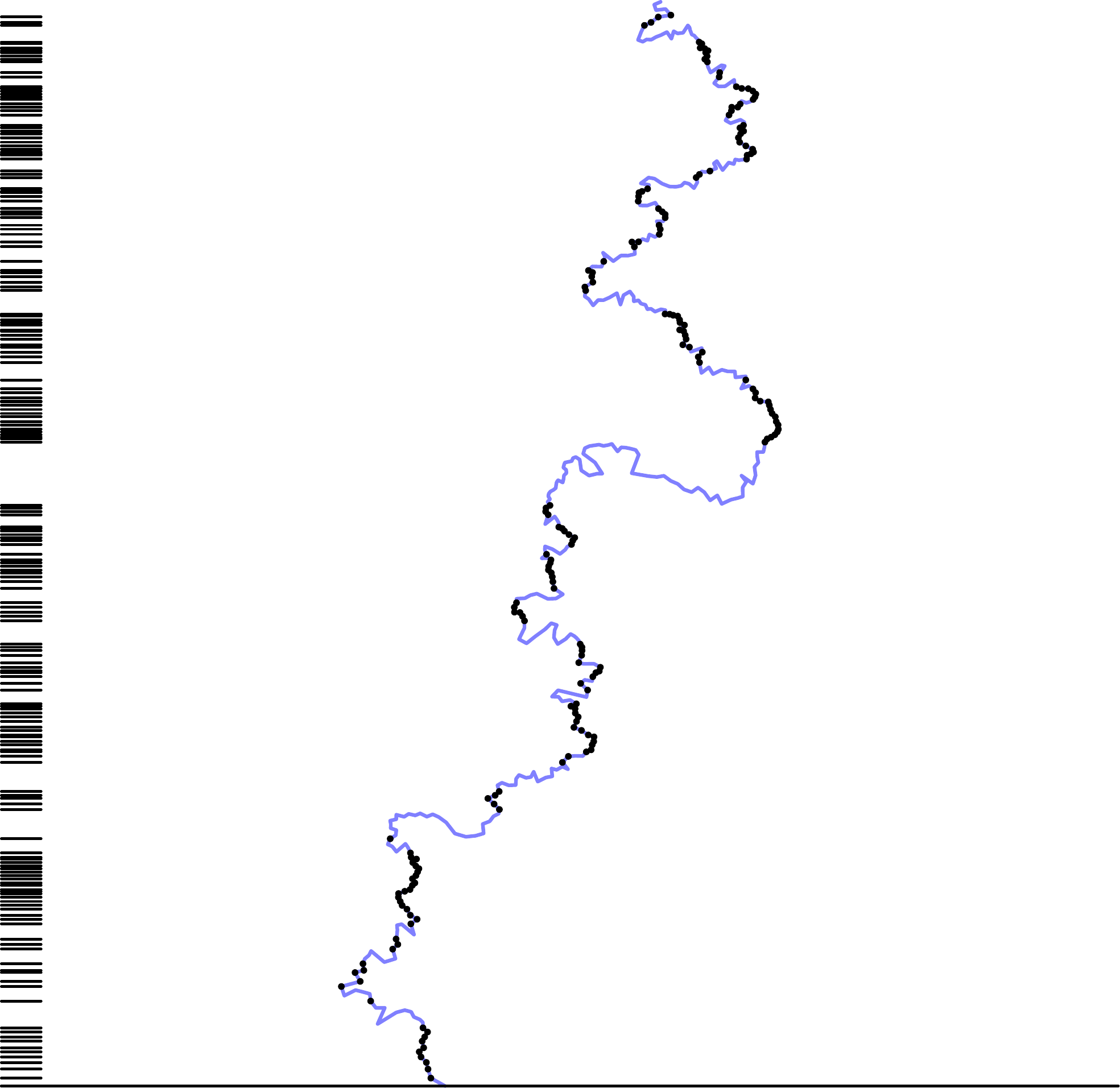}
\caption{A sample SLE(8/3) curve in the lighter colour, with the bridge points superimposed in black. The bridge heights are plotted on the vertical axis. The $\SLE(8/3)$ curve is generated by Tom Kennedy's algorithm and freely available graphics program; see \cite{kennedy:algo}.} \label{points-fig}
\end{center}
\end{figure}

\subsection{Statement of Results: The Continuum Bridge Decomposition}

In the continuum we will show that an analogue of bridge times exists for the so-called \textit{restriction hulls} in $\H$, and that these times are also renewal times. Using this renewal structure, we proceed to decompose the restriction hulls into countably many continuum irreducible bridges. This continuum decomposition most closely resembles the discrete one in the case of $\SLE(8/3)$, but we will see that it also holds for more general restriction hulls with parameter $\alpha < 1$. We will give a more in-depth description of the restriction hulls in Section \ref{PrelimSection}, but provide a brief summary here.

Roughly speaking, a restriction hull is a stochastic process taking values in the space of unbounded hulls in $\H$.  An \textit{unbounded hull} is a closed, connected subset $K \subset \overline{\H}$ such that $\H \backslash K$ consists of exactly two connected components. The unbounded hulls that we will consider are closed, connected subsets of $\overline{\H}$ that connect $0$ and $\infty$, and intersect $\R$ only at zero; moreover it will be possible to time parameterize them into a growing family $(K_t, t \geq 0)$ of \textit{hulls} (closed, connected subsets $A$ of $\overline{\H}$ such that $\H \backslash A$ is simply connected with exactly one connected component) with $K_{\infty} = K$. This time parameterization is provided by the well-known construction of restriction hulls that was originally laid out in \cite{lsw:conformal_restriction} and \cite{lawler_werner:loop_soup}. Those papers show that attaching the filled-in loops from a realization of the Brownian Loop Soup to an independent SLE curve induces a restriction law on unbounded hulls in $\H$. By changing the $\kappa$ parameter for the SLE and the intensity parameter for the loop soup (in a specific way) an entire family $\palph$ of restriction measures on unbounded hulls is created. Here $\alpha$ is a real parameter with $\alpha \geq 5/8$.

The definition of a continuum bridge is motivated by the algorithm for decomposing half-plane SAWs into irreducible bridges, which essentially searches for horizontal lines that separate the future path from the past.

\begin{definition}
Let $K$ be a hull (unbounded or not).
\begin{itemize}
\item Call $L > 0$ a \textit{bridge height} for $K$ if the horizontal line $y=L$ intersects $K$ at exactly one point, i.e. if $K \cap \{ y=L \}$ is a singleton.

\item If $z \in \H$ is such a singleton then we call it a \textit{bridge point}. Let $C$ be the set of bridge points of $K$, and let $D$ be the set of bridge heights (note that $D = \{ \Im \, z : z \in C \}$).

\item Let $G$ be the set of \textit{bridge times} at which the hull is at a bridge point, which can be written as $G := \{ t \geq 0 : K_t \backslash K_{t-} \cap C \neq \emptyset \}$.

\item A \textit{continuum bridge} is a segment of the bridge between two bridge times, i.e. if $s, t \in G$ with $s < t$ then the hull $K_{t-} \backslash K_{s-}$ is a bridge. A continuum bridge is said to be \textit{irreducible} if it contains no bridge points (other than the starting and ending points).
\end{itemize}
\end{definition}

Note that bridge heights, points and times are all functions of the \textit{entire} hull $K$. A subset of $K$ is, by itself, not enough to determine $C, D$ or $G$. At any fixed time $t \geq 0$ it is possible to determine what are the bridge points of the hull $K_t$, but not which of those are bridge points of the entire hull $K_{\infty} = K$, since some of the bridge points of $K_t$ may ultimately be destroyed by the future hull as it grows.
\medbreak
There are two main steps behind the continuum bridge decomposition. The first is to show that bridge points actually exist for hulls with $\alpha < 1$, which is not a priori clear. We do this by calculating the almost sure Hausdorff dimensions of $C$ and $D$ and showing that they are strictly larger than zero (and in fact the same). Specifically we will show the following:

\begin{theorem} \label{BridgeDimension}
Suppose $K$ has the law of $\palph$, then
% I have a problem there, how do we prove that it is prefect before the refreshing property. I'm not sure we actually can. What do you think, we should forget this property in this theorem.
\begin{enumerate}
\item \label{prop:scaling} the laws of $C$ and $D$ are scale invariant (i.e. $rC \equiv C$ and $rD \equiv D$ for all $r > 0$),
\item \label{prop:perfect} $C$ and $D$ are almost surely perfect (i.e. closed and without isolated points),
\item \label{prop:constantdim} the Hausdorff dimensions of both $C$ and $D$ are constant, $\palph-a.s.$,
\item \label{prop:hdim} $\haussdim{C} = \haussdim{D} = \max (2 - 2\alpha, 0), \,\, \palph-a.s.$,
\item \label{prop:emptyC}  $C$ and $D$ are empty, $\palph-a.s.$ if and only if $\alpha\geq 1$. %it is actually for $\alpha>1$ and not $\alpha\geq 1$...
\end{enumerate}
\end{theorem}

The proof of Theorem \ref{BridgeDimension} is taken up in Section \ref{BridgeSection}, but we will mention here that the key element is the \textbf{restriction formula:}
\begin{align}\label{restrictionFormula}
\probalph{K \cap A = \emptyset}{\alpha} = \phi_A'(0)^{\alpha},
\end{align}
where $A$ is a hull that does \textit{not} contain zero, and $\phi_A$ is a conformal map from $\H \backslash A$ to $\H$ such that $\phi_A(z) \sim z$ as $z \to \infty$. Most of the proof of Theorem \ref{BridgeDimension} is based on an analysis of $\phi_A'(0)$ for a specific choice of the hull $A$.
The proof of part \eqref{prop:emptyC} builds upon the $\alpha = 1$ case, which is related to Brownian excursions, and uses the fact that the vertical component of a Brownian excursion is a Bessel-$3$ process.

Given that bridge points exist for $\alpha < 1$, the next step is to prove an analogue of the renewal theory for half-plane SAWs. In Section \ref{RenewalSection} we show that the restriction hulls have an extended Markov property with respect to the information gained by observing the hull as it grows along with the \textit{global} bridge points of $K$ as they appear, and as a corollary we show that the bridge times are actually renewal times for the hull process. In Section \ref{LocalTimeSection} we will use this Markov property and Theorem \ref{BridgeDimension} to show the existence of a ``local time'' for the time spent by a restriction hull at its bridge points, and the local time can then be used to prove:

\begin{theorem}\label{LocalTimeIntroTheorem}
There exists a local time $\lambda$ supported on bridge heights such that $\theta_{\lambda}(K_{\lambda}\setminus K_{\lambda-})$ is a Poisson Point Process, where $\theta_t$ is an operator that shifts back to the origin the part of the hull that comes after time $t$. Moreover, the local time is the inverse of a stable subordinator of index $2-2\alpha$.
\end{theorem}

The general theory of Poisson Point Processes then implies the existence of a sigma-finite measure $\nu_{\alpha}$ on continuum irreducible bridges that is the analogue of the measure $\textbf{P}$ on irreducible bridges for half-plane SAWs. In Section \ref{LocalTimeSection} we mention some basic properties of this measure. We also show that the Poisson Point Process can be used to recover the restriction hull, so that as in the discrete case, the irreducible bridges are the building blocks of the restriction hull processes.
\medbreak
We should mention that most of these ideas are similar in spirit to the excursion decomposition of a one-dimensional Brownian motion according to its zeros, as was first described by It\^{o}. In recent years, similar two-dimensional conformally invariant decompositions of this type have also been considered by Dub\'{e}dat \cite{dubedat:excursions} and Vir\'{a}g \cite{virag:beads}. They provide decompositions of unbounded hulls arising from certain variants of $\SLE(\kappa, \rho)$ and Brownian excursions, respectively, although their decompositions are at cutpoints rather than bridge points (i.e. points that, if removed from the set, would disconnect it into two pieces). Clearly bridge points are cutpoints but not vice versa, and there does not appear to be any direct relationship between our decomposition and theirs. In one sense their decompositions are more involved than ours, since their hulls refresh at cutpoints only after conformally mapping away the past, whereas our hulls refresh at bridge points after a simple shift of the future hull back to the origin. This difference is mostly cosmetic, however, and in spirit all these decompositions are quite similar.

The paper is organized as follows: in Section \ref{PrelimSection} we give the necessary background on restriction measures and introduce some notation. Section \ref{BridgeSection} is devoted to proving the existence of bridge points and Theorem \ref{BridgeDimension}, while Section \ref{RenewalSection} proves an extended Markov property and a refreshing property of the restriction hulls with respect to the filtration generated by bridge points as they appear. Section \ref{LocalTimeSection} then uses these results to prove the decomposition of Theorem \ref{LocalTimeIntroTheorem}. Finally, in Section \ref{Open} we present a series of open questions that were raised by our work.

\bigskip
\bigskip

\noindent \textbf{Acknowledgements:} We are grateful to Wendelin Werner for initially suggesting this problem to us, for many helpful and encouraging discussions along the way, and for hosting the first author at the \'{E}cole Normale Sup\'{e}rieure where most of this work was completed. We also thank Vladas Sidoravicius for hosting us at IMPA, where this work was begun, and B\'{a}lint Vir\'{a}g for some enlightening conversations. Finally, we thank an anonymous referee for some very helpful suggestions which greatly improved the presentation of this work.

\section{Restriction Measures \label{PrelimSection}}

In this section we review the basic construction and properties of restriction measures. We include no proofs but give references to the appropriate sources. For thorough overviews of the subject see \cite{lsw:conformal_restriction, lawler_werner:loop_soup, lawler:book}. The reader interested only in the bridge decomposition for $\SLE(8/3)$, and not for general restriction measures, can entirely ignore the presence of the loops in this section.

To begin with, consider a simply connected domain $D$ in the complex plane $\C$ (other than the whole plane itself) and two boundary points $z, w \in \partial D$. A \textit{chordal restriction measure} corresponding to the triple $(D, z, w)$ is a probability measure $\P^{(D, z, w)}$ on closed subsets of $\overline{D}$. The measures are supported on closed, connected subsets of $K \subset \overline{D}$ such that $K \cap \partial D = \{z, w \}$ and $D \backslash K$ has exactly two components (for the triple $(\H, 0, \infty)$ we call these sets unbounded hulls, for obvious reasons). The restriction measures satisfy the following properties, which essentially characterize them uniquely:
\begin{itemize}
\item \textbf{Restriction property:} for all simply connected subsets $D'$ of $D$ such that $D \backslash D'$ is also simply connected and bounded away from $z$ and $w$, the law of $\P^{(D, z, w)}$, conditioned on $K \subset D'$, is $\P^{(D', z, w)}$,
\item \textbf{Conformal invariance:} if $f: D \to D'$ is conformal and $K$ has $\P^{(D, z, w)}$ as its law, then $f(K)$ is distributed according to $\P^{(f(D), f(z), f(w))}$.
\end{itemize}
It turns out that for a given triple $(D, z, w)$ there is only a one-parameter family of such laws, indexed by a real number $\alpha$. We denote the law by $\palph^{(D, z, w)}$, and due to the conformal invariance property it is enough to define the restriction measure for a single triple $(D, z, w)$. The canonical choice is $(\H, 0, \infty)$, and for shorthand we will write $\palph$ for $\palph^{(\H, 0, \infty)}$. In \cite{lsw:conformal_restriction} it is shown that these restriction measures exist only if the parameter $\alpha$ satisfies $\alpha \geq 5/8$, and that the measure is supported on simple curves only if $\alpha = 5/8$. In the latter case the restriction measure is simply the $\SLE(8/3)$ law from $z$ to $w$ in $D$. For $\alpha = 1$ it turns out that the restriction measure coincides with the law of filled-in Brownian excursions in $D$ from $z$ to $w$.

For all $\alpha \geq 5/8$, one of the fundamental constructions of \cite{lsw:conformal_restriction} is that restriction measures can be realized by adding to an $\SLEk$ curve the filled-in loops that it intersects from an independent realization of the Brownian loop soup, for an appropriate choice of $\kappa$ for the curve and intensity parameter $\lambda$ for the loop soup. Let
\begin{align*}
\kappa = \frac{6}{2\alpha + 1}, \quad \lambda = (8-3\kappa)\alpha,
\end{align*}
and let $\gamma$ be a chordal $\SLEk$ and $\ls$ be an independent realization of the Brownian loop soup (in $\H$) with intensity parameter $\lambda$. The individual loops in $\ls$ will be generically denoted by $\eta$, they can be thought of as continuous curves $\eta : [0, t_{\eta}] \to \H$ such that $\eta(0) = \eta(t_{\eta})$. Throughout we will use $\gamma$ and $\eta$ to denote the curves as well as their traces, i.e. $\gamma[0, \infty)$ and $\eta[0, t_{\eta}]$, respectively. It will be clear from the context which we are referring to. Let $K$ be the hull generated by the union of $\gamma$ and all the (filled-in) $\eta \in \ls$ such that $\eta \cap \gamma \neq \emptyset$. Then \cite{lsw:conformal_restriction} (along with  \cite{lawler_werner:loop_soup}) proves that $K$ is distributed according to $\palph$.

This construction allows us to identify restriction hulls with pairs $(\gamma, \L)$, where $\gamma : [0, t_{\gamma}] \to \C$ is a continuous, simple curve and $\L$ is a set of loops. Furthermore, the curve plus loops structure gives a clean way of time parameterizing the hulls. Letting $K$ be a restriction hull, which we identify with $(\gamma, \L)$, we define $K_t$ to be the hull generated by $\gamma[0,t]$ plus the union of all filled-in loops $\eta \in \ls$ such that $\eta \cap \gamma[0,t] \neq \emptyset$. Then $(K_t)_{t \geq 0}$ is a growing family of hulls that increases to $K_{\infty} = K$. It is important for us to have such a time parameterization so that we may properly describe the renewal theory for the restriction hulls, but the particular time parameterization is not especially important since we are mostly interested in the restriction hull as a topological object. We remark that this growing family is not continuous with respect to the time parametrization, since loops are added ``all at once'', but again it does not really matter for our purposes (nevertheless, notice that the parameterization is right continuous). The only issue to point out is that the bridge points of a restriction hull will always be a subset of the underlying (simple) curve $\gamma$, and therefore to each bridge point there is a corresponding unique bridge time. Hence the set of bridge times $G$ is a well defined object.

The curve-plus-loops structure also makes it easy to define various operations on hulls. Given two pairs $(\gamma, \L)$ and $(\gamma^*, \L^*)$ with $\gamma(0) = \gamma^*(0) = 0$, their \textit{concatenation} is defined by
\begin{align*}
(\gamma, \L) \oplus (\gamma^*, \L^*) = \left(\gamma \oplus \gamma^*, \L \cup (\gamma(t_{\gamma})+\L^*)\right),
\end{align*}
where $\gamma \oplus \gamma^*$ is the usual concatenation of curves given by
\begin{align*}
\left( \gamma \oplus \gamma^* \right) (t) = \left\{
                             \begin{array}{ll}
                               \gamma(t), & 0 \leq t \leq t_{\gamma} \\
                               \gamma^*(t - t_{\gamma}) + \gamma(t_{\gamma}), & t_{\gamma} \leq t \leq t_{\gamma}+t_{\gamma^*}
                             \end{array}
                           \right.
\end{align*}
We also define a time shift for the hulls. For $t \leq s \leq t_{\gamma}$, define the curve $\gamma^{t,s}$ by $\gamma^{t,s}(t') := \gamma(t+t')$ for $0 \leq t' \leq s-t$, and let $$\L^{t,s} := \{ \eta \in \L : \eta \cap \gamma^{t,s} \neq \emptyset, \eta \cap \gamma[0,t]=\emptyset\}.$$ Then we define $\Lambda_{t,s} K := (\gamma^{t,s}, \L^{t,s})$, which is the future hull between times $t$ and $s$, and $\theta_{t,s} K := \Lambda_{t,s} K - \gamma(t)$, which shifts the future hull to start at the origin. If $s = t_{\gamma}$, which usually for us means $s = \infty$, we write $\Lambda_t$ and $\theta_t$ for these operators. In the case that $K$ is an unbounded hull in $\H$ and $t$ is a bridge time for $K$, it is easy to see that $\theta_t K$ is also an unbounded hull in $\H$. At non-bridge times $\theta_t K$ does not remain in $\H$.

Imagine a walker moving along the hull that has discovered $K_t$ at time $t$. The information that is progressively revealed to the walker is encapsulated by the filtration
\begin{align*}
\F_t := \sigma(K_s; 0 \leq s \leq t).
\end{align*}
With respect to this filtration, the following Domain Markov property is true:
\begin{align} \label{DomainMarkov}
\textrm{The conditional law of  } \Lambda_t K, \textrm{  given  } \F_t, \textrm{  is  } \palph^{(\H \backslash \gamma[0,t], \gamma(t), \infty)}.
\end{align}
This is similar to the Domain Markov property for regular SLE, where the future curve is an independent $\SLEk$ curve from $\gamma(t)$ to $\infty$ in $\H \backslash \gamma[0,t]$, except that in the case of restriction measures one also attaches to the curve the filled-in loops of an independent realization of the Brownian loop soup in the domain $\H \backslash \gamma[0,t]$. Note, however, that both the future curve and loops are sampled from the laws corresponding to the domains $\H \backslash \gamma[0,t]$, \textit{not} the laws corresponding to $\H \backslash K_t$. In short, the future curve and future loops are allowed to intersect the past loops but \textit{not} the past curve $\gamma[0,t]$.

For the domain $(\H, 0, \infty)$ recall that the restriction measures satisfy the restriction formula \eqref{restrictionFormula}:
\begin{align*}
\probalph{K \cap A = \emptyset}{\alpha} = \phi_{A}'(0)^{\alpha},
\end{align*}
where $A$ is a hull in $\H$ that is a positive distance from zero, and $\phi_A$ is a conformal map from $\H \backslash A$ onto $\H$ satisfying $\phi_A(z) \sim z$ as $z \to \infty$. In fact, specifying the above probabilities for a sufficiently large class of hulls $A$ (so-called \textit{smooth hulls}) uniquely determines $\palph$, see \cite{lsw:conformal_restriction} for a proof of this fact. For general triples $(D, z, w)$, the restriction formula is
\begin{align}\label{genRestrictionFormula}
\palph^{(D,z,w)} \left(K \cap A = \emptyset \right) = \phi_{f(A)}'(0)^{\alpha},
\end{align}
where $A$ is a hull in $D$ not containing $z$, and $f$ is a conformal map from $D$ onto $\H$ that sends $z$ to $0$ and $w$ to $\infty$.

The restriction formula will be heavily used throughout this paper. For a given hull $A$ there are various techniques from both complex analysis and probability theory that can be used to compute $\phi_A'(0)$. We will exclusively use probabilistic techniques involving Brownian motion; these are described in the next section.

\section{Bridge Lines and Bridge Points \label{BridgeSection}}

The main focus of this section is proving Theorem \r  ef{BridgeDimension}. Specifically, we establish the existence of bridge points and lines for restriction hulls with $\alpha < 1$, and also prove the non-existence for $\alpha \geq 1$.

First observe that part \eqref{prop:scaling} of Theorem \ref{BridgeDimension} is trivial. The scale invariance of $C$ and $D$ follows immediately from the scale invariance of the restriction hulls (which itself follows from the scale invariance of SLE and of the loop soup). To prove part \eqref{prop:perfect}, first recall that bridge points of a restriction hull are always on the SLE curve itself and never on a loop, and that there is always a unique bridge time corresponding to every bridge point. We refer to the end of the section for the proof.

The most involved proofs are for calculating the Hausdorff dimensions of $C$ and $D$. The computation of the Hausdorff dimensions in Theorem \ref{BridgeDimension} follows standard ``one-point'' and ``two-point'' arguments, as in, for example, \cite{alberts_sheff:dimension, beffara:curvedim, lawler:cutpoints, schramm_zhou:dimension}. The idea behind this argument is to approximate $C$ and $D$ by ``thickened'' sets $\ep{C}$ and $\ep{D}$, and then obtain estimates on the probability that a given set of points belongs to the thickened sets. A specific bound on the probability that one point belongs to the thickened set gives an upper bound on the Hausdorff dimension, and a similar bound on the probability that two points are in the thickened sets, together with the order of magnitude of the one-point estimate, gives a lower bound on the dimension. We recall the result that we will use in the remainder; throughout this paper we use the notation $f(\epsilon) \asymp g(\epsilon)$ to indicate that there exists constants $C_1$ and $C_2$ independent of $\epsilon$ such that $C_1 g(\epsilon) \leq f(\epsilon) \leq C_2 g(\epsilon)$, for all $\epsilon$ sufficiently small.

% I think it is better to add this formal property to avoid reference%% I agree%
\begin{proposition} \label{HausdorffComputation}Let $H$ be a random subset of $\C$ and $\ep{H}$ be the set of points at distance less than $\epsilon$ from $H$. Suppose that the two following conditions are fulfilled for some $s \geq 0$ and constant $c > 0$:
\begin{itemize}
\item for all $z\in \mathbb{H}$, $\prob{z \in \ep{H}} \asymp \epsilon^{s}$,
\item for all distinct $w,z \in \H$, $\prob{w,z \in \ep{H}}\leq c\epsilon^{s} \wedge c(\epsilon^{2s}/\leb{w-z}^s)$.
\end{itemize}
Then $\haussdim H \leq 2-s$ with probability one, and with some strictly positive probability we also have $\haussdim H \geq 2-s$. If $H$ is a random subset of $\R$ then the same conclusion holds with $2-s$ replaced by $1-s$.
\end{proposition}

Note that Proposition \ref{HausdorffComputation} by itself is not enough to conclude that the Hausdorff dimension of $H$ is a constant, since the lower bound only holds on some event of positive probability. In our situation we are able to conclude that the Hausdorff dimension of $C$ and $D$ is constant by using a $0$-$1$ law. The argument that follows uses the Blumenthal $0$-$1$ Law and is modified from \cite{lawler:cutpoints}.

\begin{proof}[\textbf{Proof of Theorem \ref{BridgeDimension}, part \eqref{prop:constantdim}}]
We will prove the result for $C$, a similar argument holds for $D$. For $0 \leq t \leq s$, define $C_t(s) := \{ \textrm{bridge points of } K_s \} \cap K_t$. For a fixed $d > 0$, let $W_t(s) := \{ \haussdim C_t(s) \geq d \}$. It is enough to show that $\probalph{W_{\infty}(\infty)}{\alpha} = 0$ or $1$.

First note that for fixed $s$, both the sets $C_t(s)$ and $W_t(s)$ are increasing in $t$, while for fixed $t$ they are decreasing in $s$. Defining
\begin{align*}
V_s := \bigcap_{n=1}^{\infty} W_{\frac{1}{n}}(s) = \left \{ \haussdim C_t(s) \geq d \,\,\, \forall \,\, 0 < t \leq s \right \},
\end{align*}
it follows that $V_s$ is also decreasing in $s$. For each element of the event $V_s \backslash V_{\infty}$, there exists a $t_0$ such that $0 < t_0 \leq s$ and for all $0 < t \leq t_0$,
\begin{align*}
\haussdim C_t(\infty) < d \leq \haussdim C_t(s).
\end{align*}
But this can only happen if for every $0 < t \leq t_0$, the future hull $\Lambda_s K$ destroys bridge points of $K_s$ that are in $K_t$, and since this happens for every $0 < t \leq t_0$ and $K_t \to \{ 0 \}$ as $t \to 0$, this forces that the future hull comes arbitrarily close to the real axis. But this is clearly an event of measure zero. Hence for every $s > 0$, $\probalph{V_s \backslash V_{\infty}}{\alpha} = 0$, from which it immediately follows that
\begin{align*}
\probalph{\bigcap_{n=1}^{\infty} V_{\frac{1}{n}} }{\alpha} = \probalph{V_{\infty}}{\alpha}.
\end{align*}
However, the intersection of the $V_{1/n}$ is $\F_{0+}$-measurable, and in the case of $\SLE(8/3)$ it follows that $\probalph{V_{\infty}}{5/8} = 0$ or $1$ by the Blumenthal $0$-$1$ Law, since the corresponding measure $\P_{5/8}$ is a pushforward of Wiener measure through the Loewner equation. For general $\alpha>5/8$, the same type of Blumenthal $0$-$1$ Law holds via the usual argument. Indeed, the Domain Markov property implies that $\phi_{K_t}(\Lambda_t K)$ is a restriction hull that is independent of $\F_t$, hence for $A \in \F_{0+}$ and $t > 0$ and any bounded, continuous function $f$ on hulls we have
\begin{align*}
\expect{f \left( \phi_{K_t}(\Lambda_t K) \right) \mathbf{1}_A} = \expect{f \left( \phi_{K_t}(\Lambda_t K) \right)} \probalph{A}{\alpha}
\end{align*}
Taking a limit of both sides as $t \downarrow 0$ and using the fact that $f$ is continuous and $\phi_{K_t}$ goes continuously to the identity we get that
\begin{align*}
\expect{f(K) \mathbf{1}_A} = \expect{f(K)} \probalph{A}{\alpha},
\end{align*}
which shows that $A$ is independent of all elements of $\F_{\infty}$, and therefore of itself.
\end{proof}

We now use Proposition \ref{HausdorffComputation} to prove part \eqref{prop:hdim} of Theorem \ref{BridgeDimension}. We use the following events to define our thickened sets.

\begin{definition}
For $z \in \H$ and $\epsilon > 0$, let $I(z, \epsilon)$ be the horizontal line $y = \Im \, z$ with the gap of width $2 \epsilon$ centered around $z$ removed. That is
\begin{align*}
I(z, \epsilon) := \left \{ w \in \H : \Im \, w = \Im \, z, \, \left| \Re(w-z) \right| \geq \epsilon \right \}.
\end{align*}
Define the sets $\ep{C}$ and $\ep{D}$ by
\begin{align*}
\ep{C} &:= \left \{ z \in \H : \Eze \right \}, \quad
\ep{D} := \left \{ L > 0 : \EE{n \epsilon + iL}{\epsilon} \textrm{ for some } n \in \Z \right \}.
\end{align*}
\end{definition}

\begin{figure}
\begin{center}
\includegraphics[width=15cm, height=13cm]{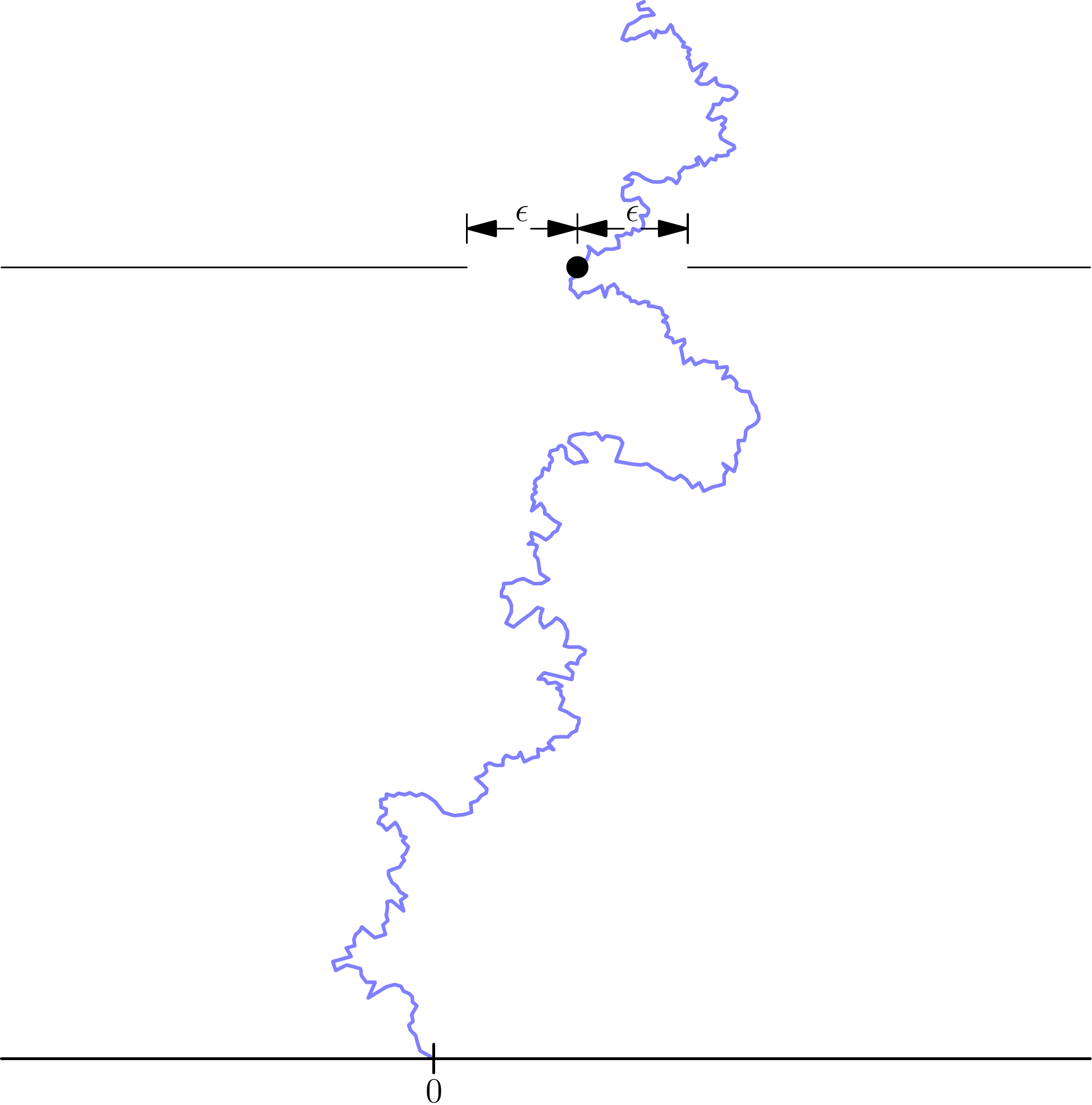}
\caption{The dotted point is $z$ and the two horizontal lines on either side form the set $I(z, \epsilon)$. This figure depicts the event that an SLE(8/3) avoids the hull $I(z, \epsilon)$.} \label{compute-fig}
\end{center}
\end{figure}

\begin{lemma}
With the definitions above, the following is true $\palph$-a.s.:
\begin{align*}
C = \bigcap_{\epsilon > 0} \ep{C}, \quad D = \bigcap_{\epsilon > 0} \ep{D}.
\end{align*}
\end{lemma}

%I added a "very" sketch proof...we must add one picture, recalling the definition and D_{\epsilon}. Can we use some pictures from your talk or shall we mak new one
\begin{proof}
Recall that $C$ consists of $z \in \H$ for which $K \cap \{ y = \Im \, z\} = \{ z \}$. Hence if $z \in C$ then $z \in \ep{C}$ for all $\epsilon > 0$. To prove the converse, note that if $z \in \ep{C}$ for every $\epsilon > 0$ then $z$ is the only possible element in the set $K \cap \{ y = \Im \, z \}$. But the latter set is always non-empty, since restriction hulls are connected and their vertical component goes from zero to infinity ($\palph$-a.s.), and therefore with $\palph$-probability $1$ the set $K \cap \{ y = L \}$ is non-empty for all $L > 0$. The proof for $D$ is exactly the same.
\end{proof}

The restriction formula makes it easy to compute the probability that a point $z \in \H$ is in $\ep{C}$. Indeed, by formula \eqref{restrictionFormula} we have
\begin{align*}
\probalph{z \in \ep{C}}{\alpha} = \probalph{ \Eze }{\alpha} = \phize'(0)^{\alpha},
\end{align*}
where $\phize$ is a conformal map from $\H \backslash I(z, \epsilon)$ onto $\H$ such that $\phize(w) \sim w$ as $w \to \infty$. Similarly,
\begin{align*}
\probalph{w,z \in \ep{C}}{\alpha} = \phi_{I(w,\epsilon) \cup I(z, \epsilon)}'(0)^{\alpha}.
\end{align*}
By Proposition \ref{HausdorffComputation}, the Hausdorff computation for $C$ and $D$ therefore comes down to an estimate of the derivative of these conformal maps at zero. We list three possible methods for these estimates. One deals only with conformal maps and is entirely analytic. The others use probabilitic techniques. We recall the analytic method but do not enter into details.

%description of methods
$ $\\
\textbf{Analytic Method:} While it is not possible to write down $\phize$ explicitly, one can write down the general form of its inverse. Let
\begin{align*}
\fze(w) := \lambda w + \frac{\Im \, z}{\pi} \left( \log(w-a) - \log(w-b) + \pi i \right),
\end{align*}
where the imaginary part of the logarithm is zero along the positive real axis and $\pi$ on the negative real axis. For appropriate choices of real constants $\lambda, a$, and $b$ (with $a < b$, $\lambda > 0$), $\fze$ maps $\H$ onto $\H \backslash I(z, \epsilon)$.  These constants implicitly depend on $z$ and $\epsilon$, although it is difficult to give closed-form expressions for them. Close analysis of the asymptotic behavior of $\lambda, a$, and $b$ could be used to get estimates on $\phize'(0)$ as $\epsilon \downarrow 0$, but we will mostly avoid this strategy. We will, however, mention that $a$ and $b$ are determined mostly by $z$, while $\lambda$ is proportional to $\epsilon^{-2}$.

$ $\\
%I have deleted the without proof and explained the developpement of both methods%
\textbf{Brownian Excursion Method:} The first probabilistic method uses a well-known formula, due to B\'{a}lint Vir\'{a}g \cite{virag:beads}, for \textit{Brownian excursions} in the upper half plane. Recall that a Brownian excursion in $\mathbb{H}$ can be thought of as a Brownian motion that is started at zero and conditioned to have a positive imaginary part at all later times. Such excursions can be realized by a random path whose horizontal component is a one-dimensional Brownian motion and whose vertical component is an independent Bessel-$3$ process.

\begin{lemma}{(\cite{virag:beads})}\label{BELemma}
Let $A$ be a compact hull in the upper half plane such that $\H \backslash A$ is simply connected and $\dist(0, A) > 0$, and $\phi_A$ be a conformal map from $\H \backslash A$ into $\H$ such that $\phi_A(0) = 0$ and $\phi_A(z) \sim z$ as $z \to \infty$. If $BE$ denotes the path of a Brownian excursion in $\H$ from $0$ to $\infty$, then
\begin{align*}
\phi_A'(0) = \prob{BE \textrm{ does not intersect } A}.
\end{align*}
\end{lemma}
In particular, this lemma shows that the filling in of a Brownian excursion has the law of a restriction measure with index $1$. It can also be used to get the estimates of Proposition \ref{HausdorffComputation}, but we prefer the following method that produces asymptotic results (even if they are not necessary in our setting).

$ $\\
\textbf{Brownian Motion Method:} Instead of using Brownian excursions to compute $\phi_A'(0)$, one can use Brownian motion directly. Oftentimes this is easier as it doesn't require dealing with the conditioning. In an appropriate sense, $\phi_A'(0)$ is the exit density at zero (with respect to Lebesgue measure) of a Brownian motion in $\H \backslash A$, starting from $\infty$. This is also called the \textit{excursion Poisson kernel} as seen from $\infty$. In what follows we let $B$ be a complex Brownian motion.

\begin{definition}
Given a simply connected domain $D$ with $z \in D$, $w \in \partial D$, let $H_D(z,w)$ denote the Poisson kernel. In the case $D = \H \backslash A$, we will often be interested in the ``Poisson kernel as seen from infinity'', for which we introduce the notation
\begin{align*}
H_{\H \backslash A}(\infty, w) := \lim_{L \uparrow \infty} L H_{\H \backslash A}(iL, w).
\end{align*}
\end{definition}
The following estimates will be useful when using Lemma \ref{BMLemma} to estimate $\phi_A'(0)$. For $x > 0$, $H_{\H}(z, x) = \frac{1}{\pi} \Im(z)/|z-x|^2$ and consequently $H_{\H}(\infty, x) = \frac{1}{\pi}$. Recall that under a conformal map $f : D \to D'$, $H_D(z,w)$ changes according to the scaling rule $H_D(z,w) = |f'(w)| H_{f(D)}(f(z), f(w)).$ In particular, we have the scaling rule $H_{\H \backslash A}(\infty, w) = H_{\H \backslash rA}(\infty, rw).$

\medbreak
The next lemma outlines how to use Brownian motion directly to estimate $\phi_A'(0)$. The method of proof is virtually identical to the one for Lemma \ref{BELemma}, so we refer the reader to \cite{virag:beads} for details.

\begin{lemma}\label{BMLemma}
For a complex Brownian motion and a compact hull $A$ in the upper half-plane such that $\H \backslash A$ is simply connected and $\dist(0,A) > 0$,
\begin{align*}
\phi_A'(0) = H_{\H \backslash A}(\infty, 0).
\end{align*}
\end{lemma}

The computation of $\phize'(0)$ is thus reduced to some estimates on the exit density of a Brownian motion in the domain $\H \backslash I(z, \epsilon)$. In order to simplify the computations, we first estimate exit densities for an intermediate set $\ep{S}$.

\begin{lemma}\label{BMStripLemma}
Let $\ep{S} = \R \times [0, 2i] \backslash I(i, \epsilon)$. Then for $x \in \R$ and $\lambda \in [-1,1]$,

\begin{align}\label{StripPoissonKernel}
H_{\ep{S}}(\lambda \epsilon + i, x) \sim \frac{\pi \sqrt{1 - \lambda^2}}{ 8 \cosh^2 ( \pi x/2) } \epsilon
\end{align}
as $\epsilon \downarrow 0$, where ``$\sim$" means that the ratio of the two terms converges to $1$ uniformly with respect to $x$ and $\lambda$. In particular, the probability that the Brownian motion started at $i$ exits $S_{\epsilon}$ on $\R$ is of order $\epsilon$.
\end{lemma}

\begin{proof}
Let $\ep{z} = \lambda \epsilon + i$. In this case, it is easy to find an explicit conformal map from $\ep{S}$ onto $\H$. A simple one is given by
\begin{align*}
\ep{f}(z) = \left(\frac{e^{\pi z}+e^{\pi \epsilon}}{e^{\pi z}+e^{-\pi \epsilon}}\right)^{1/2}.
\end{align*}
By the scaling rule for the Poisson kernel
\begin{align*}
H_{\ep{S}}(\ep{z}, x) &= |\ep{f}'(x)| H_{\H}( \ep{f}(\ep{z}), \ep{f}(x)) = \frac{|\ep{f}'(x)|}{\pi} \frac{\Im(\ep{f}(\ep{z}))}{|\ep{f}(\ep{z}) - \ep{f}(x)|^2}.
\end{align*}
It is straightforward to verify that
\begin{align*}
\ep{f}(x) &\sim 1,
\end{align*}
as $\epsilon \downarrow 0$, and
\begin{align*}
\leb{\ep{f}'(x)} &= \frac{1}{2 f_{\epsilon}(x)} \frac{2\pi e^{\pi x} \sinh(\pi \epsilon)}{(e^{\pi x} + e^{-\pi \epsilon})^2} \\
&\sim \frac{\pi^2 \epsilon}{4 \cosh^2(\pi x/2)}
\end{align*}
Similarly
\begin{align*}
\ep{f}(\ep{z}) &= \left(\frac{e^{\pi \epsilon}-e^{\pi \lambda \epsilon}}{e^{-\pi \epsilon}-e^{\pi \lambda \epsilon}}\right)^{1/2}\\
& \sim \left(\frac{1-\lambda}{-1-\lambda}\right)^{1/2}\\
&= i \left(\frac{1-\lambda}{1+\lambda}\right)^{1/2}.
\end{align*}
Assembling the pieces proves \eqref{StripPoissonKernel}, and then integrating \eqref{StripPoissonKernel} over $x$ proves the last statement.
\end{proof}

\begin{lemma}\label{BMStripLemma2}
Let $x \in \R$ and $\lambda \in [-1, 1]$. Then
\begin{align*}
H_{\H \backslash I(i, \epsilon)}(\lambda \epsilon + i, x) \sim H_{S_{\epsilon}} (\lambda \epsilon + i, x)
\end{align*}
as $\epsilon \downarrow 0$.
\end{lemma}

\begin{proof}
If a Brownian motion started at $\lambda \epsilon + i$ exits $\ep{S}$ at $x$, then it also exits $\H \backslash I(i, \epsilon)$ at $x$. Consequently, the Poisson kernel on the left hand side is bigger than the one on the right. They are not the same because the Brownian motion in $\H \backslash I(i, \epsilon)$ can hit the line $y = 2i$ before hitting zero, which the Brownian motion in $\ep{S}$ is not allowed to do. Asymptotically this event contributes nothing; indeed there is only an $O(\epsilon)$ chance that the Brownian motion even makes it up to $y = 2i$, and then another $O(\epsilon)$ chance that it passes back through the gap. Overall this makes the event of order $\epsilon^2$ (uniformly in $x$ and $\lambda$), which, by Lemma \ref{BMStripLemma}, is negligible compared to $H_{S_{\epsilon}}(\lambda \epsilon + i, x)$.
\end{proof}

\begin{proposition}\label{1PointBound}
For $z = y(x+i) \in \H$,
\begin{align*}
\phize'(0) \sim U(z) \epsilon^2
\end{align*}
as $\epsilon \downarrow 0$, where
\begin{align*}
U(y(x+i)) = \frac{\pi}{16 y^2 \cosh^2(\pi x/2)}.
\end{align*}
\end{proposition}

\begin{proof}
It suffices to prove the result in the case $z = x + i$, for the general form use the scaling rule. We use Brownian motion coming down from infinity as in Lemma \ref{BMLemma}. In order to reach $0$, the Brownian motion coming down from infinity must first pass through the gap of width $2 \epsilon$ centered at $z$, and then from the gap it must transition to zero while avoiding $I(z, \epsilon)$. The two events are independent by the Strong Markov property, and each one is $O(\epsilon)$. More precisely, by Lemmas \ref{BMStripLemma} and \ref{BMStripLemma2},
\begin{align*}
\phize'(0) &= H_{\H \backslash I(x+i, \epsilon)}(\infty, 0) \\
&= \int_{[-\epsilon,\epsilon]} H_{\H}(\infty, x+y)H_{\H \backslash I(x+i, \epsilon)}(x+y+i, 0) \, dy \\
&= \int_{-\epsilon}^{\epsilon} \frac{1}{\pi} H_{\H \backslash I(i, \epsilon)}(y+i, -x) \, dy \\
&= \frac{\epsilon}{\pi} \int_{-1}^1 H_{\H \backslash I(i, \epsilon)}(\lambda \epsilon + i, -x) \, d \lambda \\
&\sim \frac{\epsilon^2}{8 \cosh^2(\pi x/2)} \int_{-1}^1 \sqrt{1-\lambda^2} \, d \lambda .
\end{align*}
\end{proof}

From Proposition \ref{1PointBound} and the restriction formula, it is easy to derive the probability that a bridge point is within distance $\epsilon$ of a given point $z$ decays like $\epsilon^{2\alpha}$. From this the first part of Proposition \ref{HausdorffComputation} follows easily, but we need a last proposition in order to derive the two point estimate.

\begin{proposition}\label{2PointBound}
Let $z, w \in \H$, with $\Im(z) > \Im(w)$, and $\epsilon_z, \epsilon_w > 0$. Let $A = I(z, \epsilon_z) \cup I(w, \epsilon_w)$. Then
\begin{align*}
\phi_A'(0) \asymp U(z-w)U(w) \epsilon_z^2 \epsilon_w^2 ,
\end{align*}
as $\epsilon_z, \epsilon_w \downarrow 0$.
\end{proposition}

\begin{proof}
The argument is virtually the same as for the one-point estimate in Proposition \ref{1PointBound}, the only difference being that the Brownian motion, after passing through the first gap at $z$ then has to pass through a second gap at $w$. The probability of the latter event can be estimated using Proposition \ref{1PointBound}; indeed, after temporarily shifting $w$ to zero, there is a $U(z-w) \epsilon_z^2 \epsilon_w$ chance that the Brownian motion hits in an $\epsilon_w$ neighbourhood of $w$ (and therefore also the second gap). With some positive probability it hits in the middle of the second gap, where the probability of moving to zero is, up to a constant, given by $U(w) \epsilon_w$. These two probabilities multiply since, by the Strong Markov property, the path before the second gap is independent of the path after the second gap.
\end{proof}

\begin{remark}
By carefully decomposing the path according to the points it passes through in the gaps and then integrating, the statement of Proposition \ref{2PointBound} could be strengthened to an asymptotic result rather than just up to constants. For our purposes, however, this is not required.
\end{remark}

\begin{proof}[\textbf{Proof of Theorem \ref{BridgeDimension}, part \eqref{prop:hdim}}]
Propositions \ref{1PointBound} and \ref{2PointBound} combine with Proposition \ref{HausdorffComputation} to prove the result for $C$.

For $D$, the key observation is that if two gaps on a horizontal line do not overlap, then the curve can only avoid the line by going through one of them. Consequently, for $n \neq m$, the events $\EE{n\epsilon+iL}{\epsilon/2}$ and $\EE{m\epsilon+iL}{\epsilon/2}$ are disjoint, and therefore
\begin{align*}
\probalph{L \in \ep{D}}{\alpha} &= \probalph{\bigcup_{n \in \Z} \left \{ \EE{n\epsilon + iL}{\epsilon/2} \right \} }{\alpha} \\
& = \sum_{n \in \Z} \probalph{ \EE{n\epsilon+iL}{\epsilon/2} }{\alpha} \\
& \sim \frac{1}{L^{2\alpha}}  \sum_{n \in \Z} U \left( \frac{n \epsilon}{L} + i \right)^{\alpha} \left( \frac{\epsilon^{2\alpha}}{4^{\alpha}} \right)\\
& \sim \frac{\epsilon^{2\alpha-1}}{4^{\alpha} L^{2\alpha-1}} \int_{\R} U(x + iL)^{\alpha} \, dx \\
& \sim \frac{\pi^{\alpha} \epsilon^{2 \alpha - 1}}{32^{\alpha} L^{2 \alpha - 1}} \int_{\R} \cosh^{-2\alpha} \left( \pi x/2 \right) \, dx.
\end{align*}
The transition from sum to integral is a Riemann sum approximation. By $2 \alpha > 1$, the integral is a finite constant depending only on $\alpha$. This gives the one-point estimate for $D$.

Similarly, for $0 < L < L'$,
\begin{align*}
\probalph{L, L' \in \ep{D}}{\alpha} &= \probalph{\bigcup_{m,n \in \Z} \left \{ n\epsilon + iL, m\epsilon + iL' \in C_{\epsilon/2} \right \} }{\alpha} \\
&= \sum_{m,n \in \Z} \probalph{n \epsilon + iL, m \epsilon + iL' \in C_{\epsilon/2}}{\alpha} \\
&\asymp \sum_{m,n \in \Z} \epsilon^{4 \alpha} U \left[ (m-n)\epsilon + i(L'- L) \right]^{\alpha} U(n\epsilon + iL)^{\alpha} \\
&\asymp \epsilon^{4 \alpha - 2} \int_{\R} U(x + i(L'-L))^{\alpha} \, dx \int_{\R} U(x + iL)^{\alpha} \, dx \\
&\asymp \frac{\epsilon^{4 \alpha - 2}}{L^{2 \alpha - 1} (L'- L)^{2 \alpha - 1}}
\end{align*}
We use the same transition from sum to integral as in the one-point bound. Proposition \ref{HausdorffComputation} now completes the proof.
\end{proof}

We show that $C$ and $D$ are almost surely empty for $\alpha \geq 1$. For $\alpha<1$, the Haussdorff dimension is strictly positive and the set is non empty.

\begin{proof}[\textbf{Proof of Theorem \ref{BridgeDimension}, part \eqref{prop:emptyC}}]  For $\alpha=1$, recall that the imaginary part of a Brownian excursion is a Bessel(3) process, and a bridge height for the hull necessarily corresponds to a point of increase for the Bessel(3) process. However, it is well known that Bessel(3) has no point of increase since, for example, a Bessel(3) process reversed from its last passage time of a level has the same law as a Brownian motion up to its first hitting time of zero, and Brownian motion is known to have no points of increase (see \cite{revuz_yor} for details of both facts). 

For $\alpha > 1$ consider the rectangle $R = [-1,1] \times [1/2, 1]$. Cover it with $2^{2n}$ squares each of side length $2^{-n}$, and let $\{S_i\}_{1 \leq i \leq 2^{2n}}$ be the boxes and $z_i$ be their centers. Then, by Proposition \ref{1PointBound}, the expected number of squares containing a bridge point decays exponentially fast since
\begin{align*}
\mathbb{E}_{\alpha} \left[ \sum_{i=1}^{2^{2n}} \indicate{C \cap S_i \neq \emptyset} \right] & \asymp \sum_{i=1}^{2^{2n}} \probalph{I(z_i, 2^{-n}) \cap K \neq \emptyset}{\alpha} \\
& \asymp \sum_{i=1}^{2^{2n}} U(z_i) (2^{-n})^{2\alpha} \\
&= 2^{(2-2\alpha)n} 2^{-2n} \sum_{i=1}^{2^{2n}} U(z_i) \\
& \leq C 2^{(2-2\alpha)n},
\end{align*}
for some constant $C > 0$. The last inequality is a simple consequence of the fact that $U$ is Riemann integrable and hence
\begin{align*}
2^{-2n} \sum_{i=1}^{2^{2n}} U(z_i) \to \int_{R} U(z) \, dA(z) < \infty,
\end{align*}
where $dA(z)$ is two-dimensional Lebesgue measure. The Borel-Cantelli lemma then proves that $R$ almost surely contains no bridge points. By scale invariance any scaled version of $R$ also contains no bridge points. Translates of $R$ in the horizontal direction also contain no bridge points, since clearly the expected number of bridge points in translates of $R$ decreases as the rectangle is moved away from the imaginary axis. Finally, since the entire half-plane can be covered with countably many scaled and translated versions of $R$, the entire plane must almost surely be free of bridge points.
\end{proof} 

We end this section with the proof of part \eqref{prop:perfect} of Theorem \ref{BridgeDimension}. The lack of isolated points in $C$ and $D$ is also a consequence of the renewal property of restriction hulls at bridge points, so we defer the proof of this fact until the end of Section \ref{RenewalSection}.

\begin{proof}[\textbf{Proof of Theorem \ref{BridgeDimension}, part \eqref{prop:perfect}}] We prove the result for $D$; the proof for $C$ is similar. To prove that $D$ is closed, suppose that $L$ is a limit point of $D$. Without loss of generality we may assume that the limiting sequence of bridge heights $L_n$ that converges to $L$ is strictly increasing. If $t$ is the bridge time corresponding to $L$, then the restriction hull after time $t$ must reside in the domain ${\Im \, z \geq L}$ (since each $L_n$ is a bridge height). Then $L$ is not in $D$ if and only if the future hull touches the line $\Im \, z = L$ but does not cross it, which is clearly an event of probability zero. Indeed, for two points $z$ and $w$ on the same horizontal line let us define $A(z, \epsilon_z, w, \epsilon_w)$ to be the event that the hull goes through the balls $B(z,\epsilon_z)$ and $B(w,\epsilon_w)$ while avoiding $I(z,\epsilon_z)\cap I(w,\epsilon_w)$. The estimates of Proposition \ref{2PointBound} can be used to show that the probability of $A(z,\epsilon_z,w,\epsilon_w)$ is of order $\epsilon_z^{2\alpha} \epsilon_w^{2\alpha}$, which easily implies the result since $\alpha > 1/2$.
\end{proof}

\section{Renewal at Bridge Lines \label{RenewalSection}}

In this section we show that the restriction hulls renew themselves at bridge heights. Most of the section is technical, so first we would like to give the intuition behind the renewal property. It is almost entirely a consequence of restriction. Suppose that $K$ is a restriction hull with the law $\palph$. Given $\F_t$, the Domain Markov property \eqref{DomainMarkov} says that the future hull has the restriction law corresponding to the domain $(\H \backslash \gamma[0,t], \gamma(t), \infty)$. But if we also know that $t$ is a bridge time, then the future hull is separated from the past by the bridge line that the hull is currently at. The future hull is therefore conditioned not to go below this bridge line, and this conditioning is, by the restriction property, ``equivalent'' to sampling the future hull from the restriction measure corresponding to the half plane above the bridge line. Shifting the bridge point back to the origin, this means that the shifted future hull $\theta_t K$ also obeys the law $\palph$ and is independent of $\F_t$.

There are two main technical obstacles to this intuition. The first is that the event that $t$ is a bridge time for $K$ is not measurable with respect to $\F_t$, since the set of bridge times is a function of the entire hull. To address this problem and still have a meaningful notion of renewal, we simply expand our filtration to a larger one $\G_t$ that tells us which bridge heights of $K_t$ are also bridge heights of $K$. The second and more problematic technicality is that $t$ being a bridge time is an event of measure zero, and so conditioning on it requires some care. Theorem \ref{DecompMarkov} deals with this latter problem by showing that the restriction hulls obey a certain Domain Markov property with respect to $\G_t$, and from this concludes that they refresh themselves at $\G_t$-stopping times $\tau$ such that $\probalph{\tau \in G}{\alpha} = 1$ (recall that $G$ is the set of bridge times).

We make the following definitions:

\begin{definition}
For $t \geq 0$, let $D_t$ be the set of bridge heights of $K_t$. Note that $D_t$ is $\F_t$-measurable and $D_{\infty} = D$. Observe that $D_t \cap D$ is the set of bridge heights of $K_t$ that are also bridge heights of $K$, and $D_t \backslash D$ is the set of bridge heights of $K_t$ that are \textit{not} bridge heights of $K$. We also define
\begin{align*}
L_t := \sup D_t \cap D, \quad L_t' := \inf D_t \backslash D.
\end{align*}
Note that neither of these quantities, nor $D_t \cap D$ or $D_t \backslash D$, are $\F_t$-measurable. However, they are measurable with respect to the enlarged filtration
\begin{align*}
\G_t := \sigma \left(K_s, D_s \cap D; 0 \leq s \leq t \right).
\end{align*}
Clearly $\F_t \subset \G_t$, and in this larger filtration the bridge lines (and points, and times) of $K$ that belong to $K_t$ are measurable objects.
\end{definition}

Notice that $D_t \cap D$ is almost surely closed, and therefore $L_t$ is actually a maximum rather than a supremum (i.e. $L_t \in D_t \cap D$). Hence $L_t$ is the largest bridge height of $K_t$ that is also a bridge height of $K$. Clearly $L_t \leq L_t'$. The next result follows easily from these definitions.

\begin{proposition}\label{AlgebraEquality}
The $\sigma$-algebra $\Gt$ is generated by $K_t$ and $L_t$, i.e.
\begin{align*}
\G_t = \sigma \left(\F_t, L_t \right).
\end{align*}
\end{proposition}

\begin{proof}
Clearly $\sigma \left(\F_t, L_t  \right) \subset \G_t$, since $L_t$ is determined by $D_t \backslash D$. For the other direction, it is clear that $D_t \cap D = \{ L \in D_t :  L \leq L_t \}$. Hence $D_t \cap D$ is determined by both $D_t$  (which is itself determined by $K_t$) and $L_t$. This is sufficient because for $s < t$ we have $D_s \cap D \subset D_t \cap D$, and hence $D_s \cap D$ is the intersection of $D_s$, which is $\F_s$-measurable, and $D_t \cap D$, which we have just shown is $\sigma \left( \F_t, L_t \right)$-measurable.
\end{proof}

\begin{proposition}\label{RightIsolated}
For a fixed $t > 0$, $L_t < L_t'$ with probability one.
\end{proposition}

\begin{proof}
First observe that $t$ is almost surely not a bridge time. It is easy to see that the distance between $\gamma[t,\infty)$ and the last bridge line $\text{Im}(z)=L_t$ is strictly positive (for instance, there must exist another bridge height higher than $L_t$, and between, it is a continuous compact curve). But a bridge height for $\gamma[0,t)$ that is not a bridge height for the whole curve must be greater than $\inf \text{Im}(\gamma[t,\infty))$. We deduce that $L'_t$ is strictly greater than $L_t$.
\end{proof}

\begin{definition}
Given a subset $K$ of $\C$, define $J(K) := \inf \left \{ \Im \, z : z \in K \right \}$.
\end{definition}

With this definition in hand we state the paper's main technical theorem.

\begin{theorem}\label{DecompMarkov}
Suppose $K = (\gamma, \L)$ obeys the law $\palph$, and let $\tau$ be a $\Gt$-stopping time. On the event that $\tau$ is a bridge time the $\mathcal{G}_{\tau}$-conditional law of $\theta_{\tau}K$ is simply the law of a restriction hull in $\H$. If $\tau$ is not a bridge time then the conditional law of $\Lambda_{\tau}K$, given $\G_{\tau}$, is the same as the law of a restriction hull $K'$ in $\H \backslash \gamma[0,\tau]$ whose distribution is the restriction measure corresponding to the triple $(\H \backslash \gamma[0,\tau], \gamma(\tau), \infty)$, but further conditioned on the event $L_{\tau} < J(K') \leq L_{\tau}'$.
\end{theorem}

\begin{remark}
Note that if $\tau$ is a bridge time then $L_{\tau} = \Im \gamma(\tau)$ and $L_{\tau'} = \infty$. In this situation the notation $L_{\tau} < J(K') < L_{\tau'}$ can be interpreted as meaning that the future hull lies strictly above the bridge line, which is an event of measure zero. To fully emphasize this very important point we have handled this case with a separate statement at the beginning of the theorem.
\end{remark}

Theorem \ref{DecompMarkov} should be seen as the extension of the Domain Markov property \eqref{DomainMarkov} to the enlarged filtration $\G_t$. In words, it simply says that the extra information in $\G_{\tau}$ forces the future restriction hull to go below the horizontal line $y = L_{\tau}'$ but stay above the horizontal line $y = L_{\tau}$. This extra conditioning stops $L_{\tau}'$ from being a bridge height for $K$ but preserves $L_{\tau}$ as a bridge height. A detailed proof of the theorem follows. It uses a standard procedure, which we modified from \cite{virag:beads}, to bootstrap from the easy case of $\tau$ being a deterministic time to the general case that $\tau$ is a stopping time.

\begin{proof}
To simplify notation, we will write
\begin{align*}
\palph^t := \palph^{(\H \backslash \gamma[0,t], \gamma(t), \infty)} \left( \, \cdot \, \left| L_{t} < J(K') \leq L_{t}' \right. \right)
\end{align*}
throughout this proof. The goal of the proof is to show that the $\G_{\tau}$-conditional law of $\Lambda_{\tau} K$ is $\palph^{\tau}$.

Consider first the case that $\tau$ is a deterministic time $t$. Recall that conditioning on $\Gt$ is the same as conditioning on $\F_t$ and $L_t$, by Proposition \ref{AlgebraEquality}. Conditional on $\F_t$, the Domain Markov property \eqref{DomainMarkov} says that $\Lambda_t K$ has the restriction law for the triple $(\H \backslash \gamma[0,t], \gamma(t), \infty)$. Conditioning again on $L_t$ forces the future hull to stay above $y = L_t$ but to go below $y = L_t'$, and since $L_t < L_t'$ with positive probability this conditioning is well-defined. Hence the law conditioned on $\G_t$ is exactly $\palph^t$.

Another way of stating the above is as follows: let $X$ be a bounded, continuous\footnote{The topology we consider is close to the Caratheodory topology and has been defined in \cite[Lemma 3.5]{lsw:conformal_restriction}} function on hulls. Then
\begin{align}\label{MarkovEquiv}
\condsubexpect{X(\Lambda_t K)}{\Gt}{\alpha}{} = \subexpect{X}{\alpha}{t},
\end{align}
where $\E_{\alpha}$ and $\E_{\alpha}^t$ denote expectations with respect to $\palph$ and $\palph^t$, respectively. To finish the proof we need to extend \eqref{MarkovEquiv} to $\Gt$-stopping times instead of just fixed times. First suppose that $\tau$ only takes values in some countable set $\mathcal{T}$. Then
\begin{align*}
\condsubexpect{X(\Lambda_{\tau} K)}{\mathcal{G}_{\tau}}{\alpha}{} &= \sum_{t \in \mathcal{T}} \condsubexpect{X(\Lambda_{\tau} K) \indicate{\tau=t}}{\mathcal{G}_{\tau}}{\alpha}{} \\
&= \sum_{t \in \mathcal{T}} \condsubexpect{X(\Lambda_t K) \indicate{\tau = t}}{\Gt}{\alpha}{} \\
&= \sum_{t \in \mathcal{T}} \indicate{\tau = t} \condsubexpect{X(\Lambda_t K)}{\Gt}{\alpha}{} \\
&= \sum_{t \in \mathcal{T}} \indicate{\tau = t} \subexpect{X}{\alpha}{t} \\
&= \subexpect{X}{\alpha}{\tau}.
\end{align*}
From this we can bootstrap up to the case of general $\tau$. Let $\tau_n$ be the smallest element of $2^{-n} \N$ that is greater than or equal to $\tau$. Then the last argument applies to $\tau_n$, so that
\begin{align}\label{DiscreteX}
\condsubexpect{X(\Lambda_{\tau_n} K)}{\mathcal{G}_{\tau_n}}{\alpha}{} = \subexpect{X}{\alpha}{\tau_n}.
\end{align}
However, since $\tau_n$ is determined at time $\tau$ (i.e. $\tau_n$ is $\G_{\tau}$-measurable),
\begin{align*}
\condsubexpect{X(\Lambda_{\tau_n} K)}{\mathcal{G}_{\tau_n}}{\alpha}{} = \condsubexpect{X(\Lambda_{\tau_n} K)}{\mathcal{G}_{\tau}}{\alpha}{}.
\end{align*}
Since $\Lambda_{\tau_n} K \to \Lambda_{\tau} K$ as $n \to \infty$, and $X$ is bounded and continuous, it follows that the left hand side of \eqref{DiscreteX} converges to
\begin{align*}
\condsubexpect{X(\Lambda_{\tau} K)}{\mathcal{G}_{\tau}}{\alpha}{}.
\end{align*}
Hence, if we can show that $\subexpect{X}{\alpha}{\tau_n}$ converges to $\subexpect{X}{\alpha}{\tau}$ then we are done. Since $X$ is bounded and continuous, this is equivalent to showing that almost surely the law $\palph^{\tau_n}$ converges weakly to $\palph^{\tau}$, which we prove in the next lemma.
\end{proof}

\begin{lemma}
Let $\tau$ be a $\G_t$-stopping time and $\tau_n$ be the smallest element of $2^{-n} \N$ that is greater than or equal to $\tau$. Then $\palph^{\tau_n}$ converges weakly to $\palph^{\tau}$ with probability one, where we define $\palph^{\tau} \left( \cdot \right) := \palph \left( \theta_{\tau} \cdot \right)$ in the case that $\tau$ is a bridge time.
\end{lemma}

\begin{proof}
Throughout this proof we will let $H_t := (\H + iL_t) \backslash \gamma[0,t]$.

As shown in \cite[Lemma 3.2]{lsw:conformal_restriction}, a probability measure on unbounded hulls in the plane is uniquely determined by the collection of probabilities
\begin{align*}
\prob{ K \cap A = \emptyset}
\end{align*}
that is indexed by a sufficiently large class of hulls $A$. Hence it is enough to show that
\begin{align}\label{probConvergence}
\palph^{\tau_n} \left(K' \cap A = \emptyset \right) \to \palph^{\tau} \left( K' \cap A = \emptyset \right)
\end{align}
for all hulls $A$ in this class, with probability one. In our case, it is sufficient to prove that for each fixed restriction hull in $\H$, the convergence \eqref{probConvergence} holds for all hulls $A$ in $H_{\tau}$ that are a positive distance from $\gamma(\tau)$. Note that since $\tau_n \downarrow \tau$ and $\gamma$ is continuous, for sufficiently large $n$ one must have that $A$ is at positive distance from $\gamma(\tau_n)$ also. Hence the probabilities on both sides are well defined. We prove \eqref{probConvergence} in the two distinct cases that $\tau$ is and is not a bridge time.

$ $\\
\textsc{Case 1: $\tau$ is not a bridge time}

First observe that in the definition of $\palph^t$, the conditioning $J(K') > L_t$ forces the hull $K'$ to avoid the region $\{\Im \, z \leq L_t \}$, and by the restriction property this can equally be achieved by sampling $K'$ from the restriction measure corresponding to the triple $(H_t, \gamma(t), \infty)$. Thus we have the relation
\begin{align*}
\palph^{(\H \backslash \gamma[0,t], \gamma(t), \infty)} \left( \, \cdot \, \left| L_{t} < J(K') \leq L_{t}' \right. \right) = \palph^{(H_t, \gamma(t), \infty)} \left( \, \cdot \, \left | J(K') \leq L_{t}' \right. \right).
\end{align*}
Let $g_t$ be the conformal map from $H_t$ onto $\H$ such that $g_t(\gamma(t)) = 0$ and $g_t(z) \sim z$ as $z \to \infty$. Let $R_t := \{ z \in H_t : \Im \, z \leq L_t' \}$. Then
\begin{align*}
\palph^t \left( \cdot \right) = \palph^{(H_t, \gamma(t), \infty)} \left( \cdot \left | K' \cap R_t \neq \emptyset \right. \right).
\end{align*}
The first key observation is that for all $n$ sufficiently large we have that $L_{\tau_n} = L_{\tau}$. This equality is clear since $G$ is closed, and hence $\tau_n$ must belong to the same connected component of $G^c$ that $\tau$ belongs to, for $n$ sufficiently large. For these $n$ we have $L_{\tau_n} = L_{\tau}$. For $L_{\tau}'$ there are two distinct possibilities, which we now treat separately.

First note that necessarily $L_{\tau}' < \infty$. Indeed, the maximum of the imaginary part of $\Im \, K_{\tau}$ is always an element of $D_{\tau}$, and since $\tau$ is not a bridge time this maximum cannot be in $D$. So first consider the case that $L'_{\tau} < \Im \, \gamma(\tau)$. By formula \eqref{genRestrictionFormula}, we have that
\begin{align}
\palph^t \left( K' \cap A = \emptyset \right) &= \frac{\palph^{(H_t, \gamma(t), \infty)} \left(K' \cap A = \emptyset, K' \cap R_t \neq \emptyset \right)}{\palph^{(H_t, \gamma(t), \infty)} \left( K' \cap R_t \neq \emptyset \right)} \notag \\
&= \frac{\phi_{A_t}'(0)^{\alpha} - \phi_{A_t \cup S_t}'(0)^{\alpha}}{1 - \phi_{S_t}'(0)^{\alpha}}. \label{bigFormula}
\end{align}
where $A_t = g_t(A)$ and $S_t = g_t(R_t)$ (this is justified since neither $A$ nor $R_{\tau}$ contains $\gamma(\tau)$). Equation \eqref{bigFormula} shows that it is sufficient to prove
\begin{align}\label{threeConvergences}
\phi_{A_{\tau_n}}'(0) \to \phi_{A_{\tau}}'(0), \quad \phi_{A_{\tau_n} \cup S_{\tau_n}}'(0) \to \phi_{A_{\tau} \cup S_{\tau}}'(0), \quad \phi_{S_{\tau_n}}'(0) \to \phi_{S_{\tau}}'(0).
\end{align}
For $n$ large enough, $L_{\tau_n}' = L_{\tau}'$ since for any neighborhood of $\Im \, \gamma(\tau)$ there is an $n$ sufficiently large such that $D_{\tau_n} \backslash D_{\tau}$ is contained within this neighborhood. Since $L_{\tau}' < \Im \, \gamma(\tau)$, by making the neighborhood sufficiently small we get that $D_{\tau_n} \backslash D$ and $D_{\tau} \backslash D$ must have the same infimum; that is $L_{\tau_n}' = L_{\tau}'$. Hence, $A_{\tau_n}$ and $S_{\tau_n}$ are only decreasing as $\gamma[0,\tau_n]$ decreases, and again since $\gamma[0,\tau_n]$ is a simple curve that shrinks to $\gamma[0, \tau]$ it follows that $g_{\tau_n}$ converges uniformly to $g_{\tau}$ on all subcompacts of $H_{\tau}$, from which the convergences of \eqref{threeConvergences} follow (by Cauchy's derivative formula and the Schwarz reflection principle, see \cite{lsw:conformal_restriction}).

The second possibility is to have $L'_{\tau} = \Im \, \gamma(\tau)$. On the one hand, the conditioning on $K'$ going below $\Im(\gamma(\tau))$ is trivial so that $\palph^{\tau}=\palph^{(H_{\tau},\gamma(\tau),\infty)}$. On the other hand, $L'_{\tau_n}$ is greater than $L'_{\tau}$ so that one can strengthen the conditioning of $\palph^{\tau_n}$ by requiring that the future hull goes below $L'_{\tau}$. Since $\gamma$ is a simple curve shrinking to 0, one again has that $g_{\tau_n}$ converges uniformly to $g_{\tau}$ on all subcompacts of $H_{\tau}$, which proves that the conditioning becomes trivial.

$ $\\
\textsc{Case 2: $\tau$ is a bridge time}

In this case note that $A$ is a hull in the domain $\H + i \Im \, \gamma(\tau) = \H + i L_{\tau}$; hence it is simply a translate of a hull in $\H$. Moreover $g_{\tau}$ is simply the shift map $z \to z - \gamma(\tau)$, from which it follows that $A_{\tau} = A - \gamma(\tau)$ and $S_{\tau} = \H$. Since $\palph^{\tau}(\cdot) = \palph(\theta_{\tau} \cdot)$, proving \eqref{probConvergence} amounts to showing that
\begin{align*}
\palph^{\tau_n} \left( K' \cap A = \emptyset \right) \to \phi_{A_{\tau}}'(0).
\end{align*}
We use \eqref{bigFormula} to rewrite the left hand side. Define $U_t = \phi_{A_t}(S_t \cap A_t^c)$ so that
\begin{align*}
\phi_{A_t \cup S_t} = \phi_{U_t} \circ \phi_{A_t},
\end{align*}
from which it follows that
\begin{align*}
\phi'_{A_t \cup S_t}(0) = \phi'_{U_t}(0) \phi'_{A_t}(0).
\end{align*}
Therefore
\begin{align*}
\palph^{\tau_n} \left( K' \cap A = \emptyset \right) = \phi'_{A_{\tau_n}}(0)^{\alpha} \frac{1 - \phi'_{U_{\tau_n}}(0)^{\alpha}}{1 - \phi'_{S_{\tau_n}}(0)^{\alpha}}.
\end{align*}
The convergence of $\phi_{A_{\tau_n}}'(0)$ to $\phi_{A_{\tau}}'(0)$ is simple since it only involves the map $g_{\tau_n}$. Note that $L_{\tau} \leq L_{\tau_n} \leq \Im \, \gamma(\tau_n)$, so that the domains $H_{\tau_n}$ converge to $H_{\tau}$, and since $\gamma$ is a simple curve it once again follows that $g_{\tau_n}$ converges uniformly to $g_{\tau}$ on all subcompacts of $A_{\tau}$. As before, this implies the convergence of $\phi_{A_{\tau_n}}'(0)$ to $\phi_{A_{\tau}}'(0)$.

It remains to be shown that, as $n \to \infty$,
\begin{align*}
\frac{1 - \phi'_{U_{\tau_n}}(0)^{\alpha}}{1 - \phi'_{S_{\tau_n}}(0)^{\alpha}} = \frac{\palph \left ( K'' \cap U_{\tau_n} \neq \emptyset \right )}{\palph \left( K'' \cap S_{\tau_n} \neq \emptyset \right)} \to 1.
\end{align*}
Observe that
\begin{align*}
\palph \left( K \cap U_{\tau_n} \neq \emptyset \right) &= \palph \left( K \cap \phi_{A_{\tau_n}}(S_{\tau_n} \cap A_{\tau_n}^c) \neq \emptyset \right) \\
&= \palph^{(\H \backslash A_{\tau_n}, 0, \infty)} \left( K \cap S_{\tau_n} \neq \emptyset \right) \\
&\sim \palph^{(\H \backslash A_{\tau}, 0, \infty)} \left( K \cap S_{\tau_n} \neq \emptyset \right).
\end{align*}
The last relation follows since $g_{\tau_n}$ converges uniformly to $g_{\tau}$ on all subcompacts of $H_{\tau_n}$, to which $A$ eventually belongs, so that $A_{\tau_n}$ converges to $A_{\tau}$. Next recall that $S_{\tau_n} = g_{\tau_n} \left( R_{\tau_n} \right)$, and
$$0 < \sup \Im R_{\tau_n} \leq L_{\tau_n}' - \Im \gamma(\tau),$$ with the right hand side going to zero as $n \to \infty$. Since the distance of $A_{\tau}$ from zero is positive, for $n$ sufficiently large the probability that a restriction hull intersects $S_{\tau_n}$ is of the order of $\sup \Im R_{\tau_n}$ and dominated by hulls that intersect $S_{\tau_n}$ near zero. Since the set $S_{\tau_n}$ is the same near zero in both $\H$ and $\H \backslash A_{\tau}$, the ratio
\begin{align*}
\frac{ \palph^{(\H \backslash A_{\tau}, 0, \infty)} \left( K \cap S_{\tau_n} \neq \emptyset \right) }{ \palph \left( K \cap S_{\tau_n} \neq \emptyset \right) }
\end{align*}
tends to $1$.
\end{proof}

\begin{remark}
Theorem \ref{DecompMarkov} is most useful when $\tau$ is a bridge time, meaning it almost surely takes values in $G$. In that case $\gamma(\tau)$ is a bridge point for $K$, and the corresponding bridge line separates the future hull from the past. Shifting the future hull back to the origin by subtracting off $\gamma(\tau)$, we have the following:
\end{remark}

\begin{corollary}\label{ShiftedMarkovCorollary}
At $\Gt$-stopping times $\tau$ that almost surely take values in $G$, the shifted future hull $\theta_{\tau}K$ obeys the law $\palph$.
\end{corollary}

Corollary \ref{ShiftedMarkovCorollary} will be the key element in proving that the restriction hulls can be decomposed into a Poisson Point Process, which is the subject of the next section. Before doing that, we immediately apply the corollary to Theorem \ref{BridgeDimension}, part \eqref{prop:perfect} by showing that $C$ and $D$ almost surely have no isolated points.

\begin{proof}[\textbf{Proof of Theorem \ref{BridgeDimension}, part \eqref{prop:perfect}}]
We have already shown that $C$ and $D$ are closed, we prove that $C$ has no isolated points. Almost surely, zero is not isolated in $C$ because of the scale invariance and the fact that bridge points exist. For a rational number $r$, let $\tau_r$ be the first bridge time after time $r$. Then by the previous corollary, we deduce that the law of $\theta_{\tau_r}K$ obeys the law $\palph$. Since $\gamma(\tau_r)$ shifts to zero under $\theta_{\tau_r}$, the previous remark shows that  $\gamma(\tau_r)$ is almost surely not isolated. From these facts we deduce that the event $\{ \gamma(\tau_r)$ is not isolated in $C$ for all rational $r\}$ has probability one. If a point $\gamma(t) \in C$ were isolated then there would have to be an interval of time around $t$ which contains no other bridge times, but since this interval contains a rational time we arrive at a contradiction.
\end{proof}

\section{Local Time of the Decomposition \label{LocalTimeSection}}

In this section we will show that there exists a natural local time on the bridge heights that we use to decompose the restriction hulls into a Poisson Point Process of irreducible bridges. All the results of this section derive from the theory of subordinators and regenerative sets, which is well described in \cite{bertoin:subordinators}. We briefly recall the definition of regenerative sets, which is taken from \cite[Chapter 2]{bertoin:subordinators}.

\begin{definition}
A random subset $S$ of $[0, \infty)$ is a \textit{regenerative set} with respect to a filtration $\F_t$ if for every $s \geq 0$, conditionally on $M_s = \inf \{ t > s : t \in S\} < \infty$, the shifted set $\left( S - M_s \right) \cap [0, \infty)$ has the same law as $S$ and is independent of $\mathcal{F}_{M_s}$.
\end{definition}

Using the results of Sections \ref{BridgeSection} and \ref{RenewalSection}, we can immediately prove:

\begin{proposition}\label{DecompLocalTime}
The set $D$ of bridge heights is regenerative with respect to $\mathcal{D}_L := \sigma(D \cap [0,L])$.
\end{proposition}

\begin{proof}
Consider $L \geq 0$. Since $D$ is closed, $M_L \in D$ almost surely. Then $M_L$ is a bridge height, and the time $\tau_L$ at which the curve reaches this bridge height is a $\Gt$-stopping time taking values in $G$. By Corollary \ref{ShiftedMarkovCorollary}, the $\mathcal{G}_{\tau_L}$-law of $\theta_{\tau_L}K$ is the same as the original law of $K$. Consequently, the $\mathcal{G}_{\tau_L}$-law of $D(\theta_{\tau_L}K) = D - M_L$ is the same as the law of $D$. Since $\mathcal{D}_L \subset \mathcal{G}_{\tau_L}$ this completes the proof.
\end{proof}

%I changed a little bit, because as it was, we had the inpression that the scale invariant was sufficient. but you need the regenerative property
Proposition \ref{DecompLocalTime} proved that the set $D$ is regenerative, and consequently by \cite[Theorem 2.1]{bertoin:subordinators} it is the closure of the image of some subordinator (and the subordinator is unique up to a linear change of its time scale). On the other hand, Theorem \ref{BridgeDimension} showed that $D$ is scale invariant, and it is an easy step to deduce from this that the subordinator must be stable. Recall that there is a one-parameter family of stable subordinators, indexed by the real numbers between $0$ and $1$, and, as shown in \cite[Chapter 5]{bertoin:subordinators}, the index of a stable subordinator is the same as the Hausdorff dimension of its image. Hence we have the following:

\begin{corollary}\label{subordinatorCorollary}
Under the law $\palph$, the set $D$ is the closure of the image of a stable subordinator $(\sigma_{\lambda}, \lambda \geq 0)$ of index $2-2\alpha$.
\end{corollary}

The parameter $\lambda$ can be thought of as the local time corresponding to the subordinator. Recall that the local time for $\sigma$ is the function $\lambda : [0, \infty) \to [0, \infty)$ defined by $\lambda(s) := \inf \{ t \geq 0 : \sigma_t > s \}$, and it is well known in the subordinator literature that $\lambda$ is an increasing, continuous function which increases only on $D$. This means that if we run the restriction hulls on the $\lambda$ time scale, then the hull grows only when it is crossing bridge lines. For $\lambda \geq 0$ we define
\begin{align*}
\tau_h := \inf \{ t \geq 0 : \sup \Im(K_t) = h \},
\end{align*}
and
\begin{align*}
t(\lambda) := \tau_{\sigma_{\lambda}}.
\end{align*}
%%%%
Note that $\sigma_{\lambda}$ is the bridge height at which $\lambda$ units of local time are first accumulated, and then $t(\lambda)$ is the time, in the original parameterization of the restriction hull, at which the local time first reaches $\lambda$. It follows that $t(\lambda)$ is an increasing, right-continuous process for which the closure of its image is precisely the set of bridge times $G$. Intervals of $\lambda$ on which the process $t(\lambda)$ is flat correspond to times at which the restriction hull is between bridge heights. Using the $t(\lambda)$ time-scale, we are able to define a Poisson Point Process taking values in the space of irreducible bridges rooted at the origin. Let $\delta$ be the curve which starts and ends at zero in zero time (i.e. $\delta : \{0 \} \to \{ 0 \})$. For $\lambda \geq 0$, define $e_{\lambda}$ by
\begin{align} \label{eDecomp}
e_{\lambda} = \left\{
  \begin{array}{ll}
    \theta_{t(\lambda-), t(\lambda)}K, & t(\lambda) > t(\lambda-) \\
    \delta, & t(\lambda) = t(\lambda-)
  \end{array}
\right.
\end{align}
From this we have the following:

\begin{proposition}
$e_{\lambda}$ is an $\left( \mathcal{F}_{t(\lambda)} \right)_{\lambda \geq 0}$ Poisson Point Process on the space of irreducible bridges.
\end{proposition}

\begin{proof}
Take a subset $U$ of the set of irreducible bridges that doesn't contain $\delta$, and an interval $I := [\lambda_1, \lambda_2]$. As in \cite[Chapter XII]{revuz_yor}, one needs to show that the number of times that $e_{\lambda}$ belongs to $U$ for $\lambda \in I$ is independent of $\mathcal{F}_{t(\lambda_1)}$ and has the same law as the number of times that $e_{\lambda}$ belongs to $U$ for $\lambda \in [0, \lambda_2 - \lambda_1]$. But this is essentially a property of Corollary \eqref{ShiftedMarkovCorollary}.
\end{proof}

We denote by $\nu_{\alpha}$ the intensity measure of the Poisson Point Process $e_{\lambda}$, and we call it the \textbf{continuum irreducible bridge measure}. It conveniently encodes all the behavior of continuum irreducible bridges. For a set of irreducible bridges $E$, $\nu_{\alpha}(E)$ is simply the expected number of elements of $E$ that occur in $e[0,1]$, which may or may not be finite. For instance, if $E_L$ is the set of irreducible bridges with height greater than $L$, then a simple consequence of Corollary \ref{subordinatorCorollary} is that $\nu_{\alpha}(E_L) = c_{\alpha} L^{2\alpha - 2}$ for some fixed constant $c_{\alpha}$, and furthermore,
\begin{align} \label{condHeightLaw}
\textbf{P}_{\alpha}^L(\cdot) := \frac{\nu_{\alpha}(\cdot \cap E_L)}{\nu_{\alpha}(E_L)}
\end{align}
is exactly the law of the first irreducible bridge with height greater than $L$. To make the analogy with other well-known decompositions of stochastic processes, $\nu_{\alpha}$ is the equivalent of It\^o's measure on $1$-dimensional Brownian excursions, or Balint Vir\'{a}g's measure on $2$-dimensional Brownian Beads. Compared to half-plane SAWs, $\nu_{\alpha}$ is the analogue of the measure $\textbf{P}(\omega) = \beta^{-\leb{\omega}}$ on SAW irreducible bridges, although we point out that $\textbf{P}$ is a probability measure (by Kesten's relation), whereas $\nu_{\alpha}$ is infinite but $\sigma$-finite.

In the case of half-plane SAWs, the measure on paths is realized by concatenating together an i.i.d. sequence of irreducible bridges, each distributed according to $\textbf{P}$, and in the continuum a similar statement holds. If $(e_{\lambda})_{\lambda \geq 0}$ is a Poisson Point Process of irreducible bridges with intensity measure $\nu_{\alpha}$, then the concatenation
\begin{align*}
K = \bigoplus_{\lambda \geq 0} e_{\lambda}
\end{align*}
has the law of an index $\alpha$ restriction hull. Note, however, that we are not attempting to show that the irreducible bridges can be concatenated together in such a way as to reconstruct the sequence of growing hulls $(K_t)_{t \geq 0}$, even though this should be possible with enough care. Recall though that the time parameterization we are using for the restriction hulls is completely artificial to begin with, and therefore attempting to reconstruct it would mostly be an uninteresting and unuseful exercise.

\section{Open Questions \label{Open}}

In this final section we present some open questions that were raised by our work.

\begin{question}
What other properties of the irreducible bridge measure $\nu_{\alpha}$ can be derived?
\end{question}

Our work has essentially determined only one main property of bridges: that the distribution of their vertical height is the same as the jump distribution for a stable subordinator of index $2-2\alpha$ (up to a multiplicative constant). Ultimately we hope that much more can be said about irreducible bridges than this. It may be naturally difficult to say anything more, since even in the case of half-plane SAWs there is not much known about irreducible bridges (although in the ``off-critical'' case there are some results, see \cite[Chapter 4]{madras_slade:saw_book}). For other two-dimensional decompositions, notably Vir\'{a}g's Brownian Beads, it appears similarly difficult to say anything about the bead measure.

\begin{question}
Is there a constructive way of building irreducible bridges?
\end{question}

In the case of $\SLE(8/3)$, for example, is there a driving term for the Loewner equation that outputs irreducible bridges (perhaps with at least some specified vertical height)? And for general restriction measures with $\alpha < 1$, can some driving term for the Loewner equation be combined with the Brownian loop soup to produce irreducible bridges for restriction hulls?

\begin{question}
Is there a natural ``length'' that can be put on irreducible bridges?
\end{question}

For half-plane SAWs the length of the walk is simply the number of steps in it, and many results on SAWs are expressed in terms of this length. We expect that there is some way of defining a similar natural length on irreducible bridges, and that this length is somehow the scaling limit of the length for SAWs. However, because the irreducible bridges are fractal objects it is not an easy matter to define a non-trivial length on them. In the case of $\SLE(8/3)$ specifically, this question is closely related to the problem of the ``natural time parameterization'' for $\SLE$, which has recently been considered by Lawler and Sheffield \cite{lawler_sheff:time}. The key idea of their time parameterization is to build a length measure on the curve (that also has some other desirable properties), and then reparameterize in such a way that the length of the curve at time $t$ is $t$, as with the SAWs. Their length measure should also be a natural length measure for irreducible bridges.

\begin{question}
Is there some sort of continuous analogue of Kesten's relation?
\end{question}

This is closely related to the problem of the natural length on irreducible bridges described above. Supposing that $L(K)$ is the ``natural length'' of an irreducible bridge, and making an analogy with \eqref{kesten1}, we might expect that
\begin{align*}
\int_0^{\infty} \beta^{-l} \nu_{\alpha}\left( L(K) \in dl \right)
\end{align*}
is finite for $\beta < \mu$ but infinite for $\beta > \mu$, for some universal $\mu$, and then one can ask for the behavior at this critical $\mu$.

\begin{question}
Can the restriction hulls be time parameterized in such a way that the time parameterization also refreshes itself at bridge points?
\end{question}

Presently we are only showing that the hulls refresh themselves as sets and \textit{not} as time parameterized objects. But it is entirely plausible that there is some time parameterization which refreshes itself at bridge points along with the geometrical objects, especially considering that the counting parameterization for half-plane SAWs has this property (at each bridge point, one simply starts counting off the number of steps anew). It is possible that the natural time parameterization of Lawler and Sheffield will have this property for $\SLE(8/3)$ but it is not immediately clear that this will be the case, since their time parameterization has no way of seeing that it is currently at a bridge point and therefore is unlikely to refresh at such bridge times.

\begin{question}
Can some element of the bridge decomposition be used to prove the existence of, or at least heuristically deduce, critical exponents for half-plane SAWs or SAW bridges?
\end{question}

For example, it is conjectured that the number of $N$-step SAW bridges grows asymptotically like $N^{-\beta} u^N$ as $N \to \infty$, for the same $\mu$ as in \eqref{connectiveConstant} and some unknown constant $\beta$. Recently, Neal Madras has privately communicated to us his conjecture that $\beta = 7/16$, although this quantity was likely known beforehand in the physics literature. He uses two different methods to derive this value, the first being based purely on some heuristics for half-plane SAWs, and the other making use of the relation \eqref{condHeightLaw} and the conjecture that the scaling limit of half-plane SAWs is $\SLE(8/3)$. Being able to answer further questions of this type would be extremely helpful for studying half-plane SAWs.

\begin{question}\label{kappaQuestion}
Do bridge heights and lines exist for SLE($\kappa$) for values of $\kappa$ different from $8/3$. If so, what is the Hausdorff dimension of $C$ and $D$ and how does it depend on $\kappa$?
\end{question}

Currently we only know that at $\kappa = 0$ and $\kappa = 8/3$, the Hausdorff dimensions of $C$ and $D$ are $1$ and $3/4$, respectively (the $\kappa = 0$ result is clear from the fact that the corresponding SLE curve is a vertical line). We conjecture that the Hausdorff dimensions of $C$ and $D$ are always the same, and they are a strictly decreasing, continuous function of $\kappa$. When $\kappa = 4$ the Hausdorff dimension must certainly be zero since the SLE($4$) curve comes arbitrarily close to the real line, but we do not know if this is the smallest $\kappa$ for which the dimension is zero. We have no conjecture as to what that $\kappa$ might be, other than it is somewhere between $8/3$ and $4$.

We should briefly mention that, as a corollary of Theorem \ref{BridgeDimension}, we do have lower bounds on the Hausdorff dimension of $C$ and $D$ for $2 \leq \kappa \leq 8/3$. Since attaching loops to an SLE curve can only reduce the number of bridge points that the SLE curve has, we know

\begin{proposition}
Let $C$ and $D$ be the set of bridge points and heights for an SLE($\kappa$) curve, with $2 \leq \kappa \leq 8/3$. Then the Hausdorff dimensions of $C$ and $D$ are both almost surely constant, with $\haussdim C \geq 3 - \frac{6}{\kappa}$.
\end{proposition}

This lower bound is probably far from sharp, since it is increasing with $\kappa$ rather than decreasing. To prove that the Hausdorff dimensions of $C$ and $D$ are almost surely constant, Theorem \ref{BridgeDimension} part \eqref{prop:constantdim} can be used without modification.

\bibliographystyle{alpha}
\bibliography{../../Restrict}

\begin{thebibliography}{LSW04}

\bibitem[AS08]{alberts_sheff:dimension}
Tom Alberts and Scott Sheffield.
\newblock Hausdorff dimension of the sle curve intersected with the real line.
\newblock {\em Electron. Jour. Probab.}, 13:1166--1188 (electronic), 2008.

\bibitem[Bef08]{beffara:curvedim}
Vincent Beffara.
\newblock The dimension of the {SLE} curves.
\newblock {\em Ann. Probab.}, 36(4):1421--1452, 2008.

\bibitem[Ber99]{bertoin:subordinators}
Jean Bertoin.
\newblock Subordinators: examples and applications.
\newblock In {\em Lectures on probability theory and statistics
  ({S}aint-{F}lour, 1997)}, volume 1717 of {\em Lecture Notes in Math.}, pages
  1--91. Springer, Berlin, 1999.

\bibitem[Dub06]{dubedat:excursions}
Julien Dub{\'e}dat.
\newblock Excursion decompositions for {SLE} and {W}atts' crossing formula.
\newblock {\em Probab. Theory Related Fields}, 134(3):453--488, 2006.

\bibitem[Ken07]{kennedy:algo}
Tom Kennedy.
\newblock A fast algorithm for simulating the chordal {S}chramm-{L}oewner
  evolution.
\newblock {\em J. Stat. Phys.}, 128(5):1125--1137, 2007.

\bibitem[Kes63]{kesten:saw1}
Harry Kesten.
\newblock On the number of self-avoiding walks.
\newblock {\em J. Mathematical Phys.}, 4:960--969, 1963.

\bibitem[Kes64]{kesten:saw2}
Harry Kesten.
\newblock On the number of self-avoiding walks. {II}.
\newblock {\em J. Mathematical Phys.}, 5:1128--1137, 1964.

\bibitem[Law96]{lawler:cutpoints}
Gregory~F. Lawler.
\newblock Hausdorff dimension of cut points for {B}rownian motion.
\newblock {\em Electron. J. Probab.}, 1:no.\ 2, approx.\ 20 pp.\ (electronic),
  1996.

\bibitem[Law05]{lawler:book}
Gregory~F. Lawler.
\newblock {\em Conformally Invariant Processes in the Plane}, volume 114 of
  {\em Mathematical Surveys and Monographs}.
\newblock American Mathematical Society, Providence, RI, 2005.

\bibitem[LS09]{lawler_sheff:time}
Gregory~F. Lawler and Scott Sheffield.
\newblock Construction of the natural parameterization for {SLE} curves.
\newblock {\em arXiv:0906.3804v1 [math.PR]}, 2009.

\bibitem[LSW03]{lsw:conformal_restriction}
Gregory Lawler, Oded Schramm, and Wendelin Werner.
\newblock Conformal restriction: the chordal case.
\newblock {\em J. Amer. Math. Soc.}, 16(4):917--955 (electronic), 2003.

\bibitem[LSW04]{lsw:saw}
Gregory~F. Lawler, Oded Schramm, and Wendelin Werner.
\newblock On the scaling limit of planar self-avoiding walk.
\newblock In {\em Fractal geometry and applications: a jubilee of Beno\^\i t
  Mandelbrot, Part 2}, volume~72 of {\em Proc. Sympos. Pure Math.}, pages
  339--364. Amer. Math. Soc., Providence, RI, 2004.

\bibitem[LW04]{lawler_werner:loop_soup}
Gregory~F. Lawler and Wendelin Werner.
\newblock The {B}rownian loop soup.
\newblock {\em Probab. Theory Related Fields}, 128(4):565--588, 2004.

\bibitem[MS93]{madras_slade:saw_book}
Neal Madras and Gordon Slade.
\newblock {\em The self-avoiding walk}.
\newblock Probability and its Applications. Birkh\"auser Boston Inc., Boston,
  MA, 1993.

\bibitem[RY99]{revuz_yor}
Daniel Revuz and Marc Yor.
\newblock {\em Continuous martingales and {B}rownian motion}, volume 293 of
  {\em Grundlehren der Mathematischen Wissenschaften [Fundamental Principles of
  Mathematical Sciences]}.
\newblock Springer-Verlag, Berlin, third edition, 1999.

\bibitem[SZ07]{schramm_zhou:dimension}
Oded Schramm and Wang Zhou.
\newblock Boundary proximity of {SLE}.
\newblock arXiv:0711.3350v2 [math.PR], 2007.

\bibitem[Vir03]{virag:beads}
B{\'a}lint Vir{\'a}g.
\newblock Brownian beads.
\newblock {\em Probab. Theory Related Fields}, 127(3):367--387, 2003.

\end{thebibliography}
\end{document}